\documentclass[10pt,reqno]{amsart}
\usepackage{epsfig,amssymb,amsmath,version}
\usepackage{amssymb,version,graphicx,fancybox,mathrsfs}
\usepackage{url,hyperref}
\usepackage{subfigure}
\usepackage{color}
\usepackage{stmaryrd}
\usepackage{multirow}
\usepackage{booktabs,siunitx,enumitem}
\usepackage{booktabs}
\usepackage[ruled,vlined]{algorithm2e}
\usepackage[utf8]{inputenc}
\usepackage[english]{babel}
\usepackage{amsmath, bm}

\catcode`\@=11 \theoremstyle{plain}
\@addtoreset{equation}{section}   

\@addtoreset{figure}{section}
\renewcommand\thefigure{\@arabic\c@figure}
\newtheorem{thm}{\bf Theorem}

\newenvironment{theorem}{\begin{thm}} {\end{thm}}
\newtheorem{cor}{\bf Corollary}

\newtheorem{lmm}{\bf Lemma}

\newenvironment{lemma}{\begin{lmm}}{\end{lmm}}
\theoremstyle{remark}
\newtheorem{rem}{\bf Remark}

\def \epsilon {{\varepsilon}}

\definecolor{bgblue}{rgb}{0.04,0.39,0.54}
\definecolor{lired}{rgb}{0.3, 0.0, 0.0}
\definecolor{ligreen}{rgb}{0.0, 0.3, 0.0}
\definecolor{liblue}{rgb}{0.9, 1.0, 1.0}
\definecolor{gray}{rgb}{0.6, 0.6, 0.6}
\definecolor{sky}{rgb}{0.3, 1.0, 1.0}
\definecolor{bunhong}{rgb}{1.0, 0.3, 1.0}
\definecolor{yellow}{rgb}{0.97, 1, 0.0}
\definecolor{liyellow}{rgb}{0.9, 0.8, 0.0}
\definecolor{cengse}{rgb}{0.00,0.40,0.29}

\def \bx {\bm x}
\renewcommand \wedge \times

\begin{document}

{\title[High-order schemes and error analysis for Navier-stokes equations] {
Stability and error analysis of a class of high-order IMEX schemes  for Navier-stokes equations with periodic boundary conditions}

\author[
	F. Huang and J. Shen
	]{
	Fukeng Huang and Jie Shen
		}
	\thanks{Department of Mathematics, Purdue University. This research  is partially supported by NSF grants DMS-2012585 and AFOSR FA9550-20-1-0309. Emails: huang972@purdue.edu (F. Huang), shen7@purdue.edu (J. Shen).}

\keywords{Navier-stokes, stability, error analysis, high-order}
 \subjclass[2000]{65M15; 76D05; 65M70}

\begin{abstract}
We construct high-order  semi-discrete-in-time and fully discrete (with Fourier-Galerkin in space)    schemes  for the incompressible Navier-Stokes equations with periodic boundary conditions, and carry out corresponding error analysis. The schemes are of implicit-explicit type based on a scalar auxiliary variable (SAV) approach. It is shown that numerical solutions of these schemes   are uniformly bounded without any restriction on time step size. These uniform bounds enable us to  carry out a rigorous error analysis for the  schemes up to fifth-order  in a unified form, and derive global  error estimates in $l^\infty(0,T;H^1)\cap l^2(0,T;H^2)$ in the two dimensional case as well as local  error estimates in $l^\infty(0,T;H^1)\cap l^2(0,T;H^2)$ in the three dimensional case. We also present numerical results confirming our theoretical convergence rates and demonstrating  advantages of higher-order schemes for flows with complex structures in  the double shear layer problem.

\end{abstract}
 \maketitle

\section{Introduction}
Numerical approximation of the Navier-Stokes equations has been a subject of intensive study for many decades and continues to attract considerable attention, as it plays a fundamental role in computational fluid dynamics. Most of the work are concerned with the  Navier-Stokes equations with non periodic boundary conditions, as is the case with the most applications. An enormous amount of work has been devoted to construct efficient and stable numerical algorithms for solving the incompressible Navier-Stokes equations with non periodic boundary conditions, see \cite{girault1979finite,Tema84,deville2002high,glowinski2003finite,gunzburger2012finite,peyret2013spectral} and the references therein. In particular, the  papers
\cite{baker1982higher,heywood1990finite,weinan1995projection,guermond2006overview,He.S07,MR2475954}, among others,  are particularly concerned with the error estimates for semi-discrete-in-time or fully discrete schemes.

We consider in this paper numerical approximation of  the  incompressible Navier-Stokes equations in primitive formulation:
\begin{subequations}\label{eq:NS}
\begin{align}
& \frac{\partial \bm u}{\partial t}-\nu \Delta \bm u+( \bm u\cdot \nabla) \bm u +\nabla p=0,\\
& \nabla \cdot \bm u=0,
\end{align}
\end{subequations}
with a suitable initial condition  $\bm u|_{t=0}=\bm u_0$ in a rectangular domain $\Omega\subset \mathbb{R}^d\;(d=2,3)$ with   periodic boundary conditions.  The unknowns are velocity $\bm u$ and the pressure $p$ which is assumed to have zero mean for uniqueness,  $\nu>0$ is the viscosity. To simplify the presentation, we have set the   external force to be zero. But our schemes and analytical results can be naturally extended to the case with a non-zero external force.

The incompressible Navier-Stokes equations with periodic boundary conditions retain the essential mathematical properties/difficulties of the system with non periodic boundary conditions, but are amenable to very efficient numerical algorithms using the Fourier-spectral method, and are particularly useful in the study of homogeneous turbulence \cite{orszag1972numerical,she1991structure,moin1998direct}.

There exists also a significant  number of work  devoted to the numerical analysis for Navier-Stokes equations with periodic boundary conditions. For examples, in \cite{hald1981convergence}, Hald proved the convergence of semi-discrete Fourier-Galerkin methods in two and three dimensions; in \cite{weinan93}, E used semigroup theory to establish convergence and error estimates of the semi-discrete Fourier-Galerkin and Fourier-collocation methods  in various energy norms and $L^p$-norms; in \cite{wang2012efficient}, Wang proved  uniform bounds and  convergence of
long time statistics for a semi-discrete second-order implicit-explicit (IMEX) scheme for the 2-D Navier-Stokes equations with periodic boundary conditions in vorticity-stream function formulation, see also  related work in \cite{gottlieb2012long,tone2015long}; in \cite{cheng_wang2016}, Cheng and Wang established  uniform bounds  for  semi-discrete higher-order (up to fourth-order) IMEX scheme for the 2-D Navier-Stokes equations with periodic boundary conditions in vorticity-stream function formulation; in
\cite{heister2017unconditional}, Heister et al. proved uniform bounds for a  fully discrete finite-element and second-order IMEX scheme  for the 2-D Navier-Stokes equations with periodic boundary conditions in vorticity-velocity formulation. Note that the uniform bounds for semi-discrete  IMEX schemes obtained in the above references are for two-dimensional cases only and require that the time step be sufficiently small.

It appears that, except some recently constructed schemes based on the scalar auxiliary variable (SAV) approach \cite{lin2019numerical,li2020error}, all other IMEX type schemes (i.e., the nonlinear term is treated explicitly) for Navier-Stokes equations require the time step to be sufficiently small to have a bounded numerical solution. Furthermore,
to the best of our knowledge, there is no error analysis for any IMEX scheme  for the three-dimensional Navier-Stokes equations, and no error estimate is available for  any higher-order ($\ge 3$) IMEX scheme. 

In this paper, we construct semi-discrete and fully discrete with Fourier-Galerkin in space  SAV IMEX schemes
and carry out a unified stability and error analysis. Our main contributions include:
\begin{itemize}
 \item Our semi-discrete and fully discrete schemes of arbitrary order in time are unconditionally stable without any restriction on time step size;
 \item Global error estimates in $l^\infty(0,T;H^1)\cap l^2(0,T;H^2)$ up to fifth-order in time are established for the two-dimensional case;
 \item Local error estimates in $l^\infty(0,T_* ;H^1)\cap l^2(0,T_* ;H^2)$ (with a $T_* \le T$)  up to fifth-order in time are established for the three-dimensional case.
\end{itemize}
Our schemes are constructed using the  SAV approach proposed in \cite{HSY20}  which can be used for general dissipative systems. The main advantages of this approach, compared with other SAV approaches proposed in \cite{lin2019numerical,li2020error} for Navier-Stokes equations is that our schemes are linear, decoupled and can be high-order. Moreover, in the two dimensional case, we use a stronger energy dissipation law  \eqref{eq:diss2d}, which is only true for the 2-D Navier-Stokes equations with periodic boundary conditions, that leads to a uniform bound for the numerical solution in $l^\infty(0,T;H^1)$, as opposed to $l^\infty(0,T;L^2)$ in the three dimensional case.

The rest of the paper is organized as follows. In the next section,  we provide some preliminaries to be used in the sequel. In Section 3,
we describe our semi-discrete and fully discrete with Fourier-Galerkin SAV schemes for the Navier-Stokes equations with periodic boundary condition,  prove its unconditionally stability, and provide some numerical results to demonstrate the convergence rates and validate the robustness of our schemes. In section 4, we present  detailed error analysis for the $k$th-order schemes $(k=1,2,3,4,5)$ in a unified form. Some concluding remarks are given in the last section.

\section{Preliminaries}
We first  introduce some notations.
  We denote by $(\cdot, \cdot)$ and $\|\cdot\|$ the inner product and the norm in $L^2(\Omega)$, and denote
$$\bm H^k_p(\Omega)=\{u^j \,(j=0,1,\cdots,k)\, \in L^2(\Omega): u^j \,(j=0,1,\cdots,k-1)\, \quad \text{periodic},\quad \int_\Omega u d\bx=0\},$$
with norm $\|\cdot\|_k$. For non-integer $s>0$, $H^s_p(\Omega)$ and the corresponding norm $\|\cdot\|_s$ are defined by space interpolation \cite{adams2003sobolev}.
In particular, we set $H^0_p(\Omega)=L^2_0(\Omega)$.

Let $V$ be a Banach space, we shall also use the standard notations
  $L^p(0,T;V)$ and $C([0,T];V)$. To simplify the notation, we often omit the spatial dependence  for the exact solution $u$, i.e.,  $u(x,t)$ is often denoted by $u(t)$.  We shall use bold faced letters to denote vectors and vector spaces, and use $C$ to denote a generic positive constant independent of the discretization parameters.

We now define the following spaces which are particularly used for Navier-Stokes equations:
\begin{equation*}
\textbf{H}=\{\bm v \in \bm L_0^2(\Omega): \nabla\cdot \bm v =0\},\quad  \textbf{V}=\{\bm v \in \bm H^1_p(\Omega):  \nabla\cdot \bm v =0\}.
\end{equation*}
Let $\bm v\in \bm L^2_0(\Omega)$, we define $w:=\Delta^{-1} \bm v$  as  the solution of
\begin{equation*}
 \Delta \bm w=\bm v \quad \bm x\in \Omega;\quad \bm w \;\text{periodic with zero mean}.
\end{equation*}
Note that in the periodic  case, we can define   the operators $\nabla$, $\nabla\cdot$ and $\Delta^{-1}$
in the Fourier space by expanding functions and their derivatives in Fourier series, and one can easily show that  these operators commute with each other.

We define   a linear operator \textbf{A} in $\bm L_0^2(\Omega)$ by
\begin{equation}\label{eq:A}
\textbf{A} \bm v := \nabla \times \nabla \times \Delta^{-1} \bm v, \quad\forall \bm v\in \bm L_0^2(\Omega).
\end{equation}
Since
\begin{equation*}
\|\Delta \bm w\|^2=\|\nabla \times \nabla \times \bm w\|^2+\|\nabla\nabla \cdot \bm w\|^2\quad \forall \bm w\in \bm H^2_p(\Omega),
\end{equation*}
we derive immediately from the above that
\begin{equation}\label{eq:Anorm}
\|\textbf{A} \bm v\|= \|\Delta \Delta^{-1}\bm v\|-\|\nabla\nabla \cdot \Delta^{-1}\bm v\| \le \|\bm v\|, \,\,\forall \bm v \in \bm L_0^2(\Omega).
\end{equation}
Next, we define the trilinear form $b(\cdot,\cdot,\cdot)$  and $b_\textbf{A}(\cdot,\cdot,\cdot)$ by
\begin{equation*}
b(\bm u, \bm v, \bm w)= \int_{\Omega}(\bm u \cdot \nabla) \bm v \cdot \bm w d\bx,\,\,b_\textbf{A}(\bm u, \bm v, \bm w)= \int_{\Omega}\textbf{A}((\bm u \cdot \nabla) \bm v) \cdot \bm w d\bx.
\end{equation*}
In particular, we have
\begin{equation*}
b(\bm u, \bm v, \bm w)=-b(\bm u, \bm w, \bm v),\,\forall \bm u \in \textbf{H}, \, \bm v,\bm w \in \bm H_p^1(\Omega),
\end{equation*}
which implies
\begin{equation}\label{tri1}
b(\bm u, \bm v, \bm v)=0,\,\forall \bm u \in \textbf{H}, \, \bm v\in \bm H_p^1(\Omega).
\end{equation}
In the two-dimensional periodic case, we have also \cite{Tema83}
\begin{equation}\label{tri2}
b(\bm u, \bm u, \Delta \bm u)=0,\,\forall \bm u \in  \bm H_p^2(\Omega).
\end{equation}
Taking the inner product of \eqref{eq:NS} with $\bm u$, thanks to \eqref{tri1}, we find that solution of the Navier-Stokes equations \eqref{eq:NS} satisfies the energy dissipation law
\begin{equation}\label{eq:diss}
\frac 12\frac{d}{dt}\|\bm u\|^2=-\nu \|\nabla \bm u\|^2 \quad\quad (d=2,3).
\end{equation}
On the other hand, in the two dimensional periodic case, taking the inner product of \eqref{eq:NS} with $-\Delta \bm u$, thanks to \eqref{tri2},   we derive another energy dissipation law \cite{Tema83}
\begin{equation}
\label{eq:diss2d}
\frac{1}{2}\frac{d}{dt} \|\nabla \bm u\|^2=-\nu \|\Delta \bm u\|^2 \quad \quad (d=2) .
\end{equation}
Using \eqref{eq:Anorm}, H\"older inequality and Sobolev inequality, we have \cite{Tema83}
\begin{eqnarray}
b(\bm u,\bm v,\bm w),\; b_\textbf{A}(\bm u,\bm v,\bm w)& \le&
   c\|\textbf{u}\|_1^{1/2}\|\textbf{u}\|^{1/2}\|\textbf{v}\|_2^{1/2}\|\textbf{v}\|_1^{1/2}\|\textbf{w}\|,\;   \quad  d=2;\label{eq:ineq2d}\\
b(\bm u,\bm v,\bm w),\;b_\textbf{A}(\bm u,\bm v,\bm w) &\le& c\|\bm u\|_1\|\nabla \bm v\|_{1/2}\|\bm w\|,  \quad  d=3.\label{eq:ineq3d}
 \end{eqnarray}
We also use frequently the following inequalities \cite{Tema83}:
\begin{equation}\label{eq:ineq2}
b(\bm u,\bm v,\bm w),\;b_\textbf{A}(\bm u,\bm v,\bm w) \le
\left\{
\begin{array}{lr}
  c\|\bm u\|_1\|\bm v\|_1\|\bm w\|_1;\\
  c\|\bm u\|_2\|\bm v\|_0\|\bm w\|_1;\\
 c\|\bm u\|_2\|\bm v\|_1\|\bm w\|_0;\\
 c\|\bm u\|_1\|\bm v\|_2\|\bm w\|_0;\\
c\|\bm u\|_0\|\bm v\|_2\|\bm w\|_1;
\end{array}
\right. \quad d\le 4.
\end{equation}
Note that \eqref{tri2}, \eqref{eq:diss2d},   and \eqref{eq:ineq2d} enable us to obtain global error estimates  in the two-dimensional case.

\section{The SAV schemes and  stability results}
In this section, we construct semi-discrete and fully discrete  SAV schemes  for the incompressible Navier-Stokes equations, and establish stability results for both semi-discrete and fully discrete schemes.
More precisely,
 we shall prove uniform $L^2$ bound for the SAV scheme based on  the dissipation law  \eqref{eq:diss} in 3D case, and prove a  uniform  $H^1$   bound for the SAV scheme based on the dissipation law  \eqref{eq:diss2d} in 2D case.

\subsection{The SAV schemes}
Following the ideas in \cite{HSY20} for the general dissipative systems, we  construct below  unconditionally energy stable schemes for \eqref{eq:NS}.

For Navier-Stokes equations with periodic boundary conditions, we can explicitly eliminate the pressure from \eqref{eq:NS}. Indeed, taking the divergence on both sides of \eqref{eq:NS}, we find
\begin{equation}\label{prePoisson}
-\Delta p=\nabla\cdot ({\bm u}\cdot\nabla {\bm u}),
\end{equation}
from which we derive
\begin{equation}\label{prePoisson2}
\begin{split}
\nabla p&= \nabla \Delta^{-1} \Delta p=-\nabla \Delta^{-1} \nabla\cdot ({\bm u}\cdot\nabla {\bm u})\\
&=-\nabla \nabla\cdot  \Delta^{-1}({\bm u}\cdot\nabla {\bm u})=-(\Delta+\nabla \times \nabla \times) \Delta^{-1}({\bm u}\cdot\nabla {\bm u}) \\
&= -{\bm u}\cdot\nabla {\bm u}- \nabla \times \nabla \times \Delta^{-1}({\bm u}\cdot\nabla {\bm u}) =-{\bm u}\cdot\nabla {\bm u}- \textbf{A} ({\bm u}\cdot\nabla {\bm u}),
\end{split}
\end{equation}
 where $\textbf{A}$ is defined in \eqref{eq:A}.
Hence, \eqref{eq:NS} is equivalent to \eqref{prePoisson} and
\begin{equation}\label{eq:NS1b}
  \frac{\partial \bm u}{\partial t}-\nu \Delta \bm u-\textbf{A} (\bm u\cdot \nabla \bm u)=0.
\end{equation}

In order to apply the SAV approach, we introduce a SAV, $r(t)=E(\bm u(t))+1$, and expand \eqref{eq:NS1b}  as
\begin{subequations}\label{eq:NSb}
\begin{align}
&  \frac{\partial \bm u}{\partial t}-\nu \Delta \bm u-\textbf{A} (\bm u\cdot \nabla \bm u)=0, \label{eq:NS1}\\
& \frac{d E}{dt}=\left\{
\begin{array}{lr}
 -\nu\frac{r(t)}{E(\bm u(t))+1} \|\Delta \bm u\|^2, \qquad d=2,\\
-\nu\frac{r(t)}{E(\bm u(t))+1} \|\nabla \bm u\|^2, \qquad d=3,
\end{array}
\right.\label{eq:NS3}
\end{align}
\end{subequations}
where
\begin{equation}
 E(\bm u)=\begin{cases} \frac 12\|\nabla \bm u\|^2,& \;d=2,\\
 \frac 12\|\bm u\|^2,& \;d=3.\end{cases}
\end{equation}
We construct below semi-discrete and fully discrete schemes for the expanded system \eqref{eq:NSb}.

\subsubsection{Semi-discrete SAV schemes}
We consider first the time discretization of \eqref{eq:NSb} based on the implicit-explicit BDF-$k$ formulae in the following unified form:

Given $r^n$, $\bm u^j\;(j=n,n-1,\cdots,n-k+1)$,  we compute $\bar {\bm u}^{n+1},\, r^{n+1}, p^{n+1}, \, \xi^{n+1}$ and $\bm u^{n+1}$ consecutively by
\begin{subequations}\label{eq: NSsavb}
\begin{align}
& \frac{\alpha_k \bar{\bm u}^{n+1}-A_k(\bar{\bm u}^n)}{\delta t}-\nu \Delta \bar{\bm u}^{n+1}-\textbf{A} (B_k({\bm u}^n) \cdot \nabla B_k({\bm u}^n))=0,\label{eq: NSsav1b}\\
& \frac{1}{\delta t}\big(r^{n+1}-r^{n}\big)=\left\{
\begin{array}{lr}
-\nu\frac{r^{n+1}}{E(\bar{\bm u}^{n+1})+1} \|\Delta \bar{\bm u}^{n+1}\|^2,\qquad d=2,\\
-\nu\frac{r^{n+1}}{E(\bar{\bm u}^{n+1})+1} \|\nabla \bar{\bm u}^{n+1}\|^2,\qquad d=3;
\end{array}
\right.\label{eq: NSsav3b}\\
& \xi^{n+1}=\frac{r^{n+1}}{E(\bar{\bm u}^{n+1})};\label{eq: NSsav4b}\\
& \bm u^{n+1}=\eta_{k}^{n+1}\bar{\bm u}^{n+1} \;\text{ with } \eta^{n+1}_{k}=1-(1-\xi^{n+1})^{k}.\label{eq: NSsav5b}
\end{align}
\end{subequations}
Whenever pressure is needed, it can be computed from
\begin{equation}\label{prePoissont}
 \Delta p^{n+1}=-\nabla \cdot (\bm u^{n+1} \cdot \nabla \bm u^{n+1}).
\end{equation}
In the above,  $\alpha_k,$ the operators $A_k$ and $B_k$ $(k=1,2,3,4,5)$ are given by:
\begin{itemize}
\item[first-order:]
\begin{equation}\label{eq:bdf1}
\alpha_1=1, \quad A_1(\bm u^n)=\bm u^n,\quad B_1(\bar {\bm u}^n)=\bar {u}^n;
\end{equation}
\item[second-order:]
\begin{equation}\label{eq:bdf2}
\alpha_2=\frac{3}{2}, \quad A_2(\bm u^n)=2\bm u^n-\frac{1}{2}\bm u^{n-1},\quad B_2(\bar {\bm u}^n)=2\bar {\bm u}^n-\bar {\bm u}^{n-1};
\end{equation}
\item[third-order:]
\begin{equation}\label{eq:bdf3}
\alpha_3=\frac{11}{6}, \quad A_3(\bm u^n)=3\bm u^n-\frac{3}{2}\bm u^{n-1}+\frac{1}{3}\bm u^{n-2},\quad B_3(\bar {\bm u}^n)=3\bar {\bm u}^n-3\bar {\bm u}^{n-1}+\bar{\bm u}^{n-2};
\end{equation}
\item[fourth-order:]
\begin{equation}\label{eq:bdf4}
\alpha_4=\frac{25}{12}, \; A_4(\bm u^n)=4\bm u^n-3\bm u^{n-1}+\frac{4}{3}\bm u^{n-2}-\frac{1}{4}\bm u^{n-3},\; B_4(\bar {\bm u}^n)=4\bar {\bm u}^n-6\bar {\bm u}^{n-1}+4\bar {\bm u}^{n-2}-\bar {\bm u}^{n-3};\end{equation}
\item[fifth-order:]
\begin{equation}\label{eq:bdf5}
\begin{split}
\alpha_5=\frac{137}{60}, \quad & A_5(\bm u^n)=5\bm u^n-5\bm u^{n-1}+\frac{10}{3}\bm u^{n-2}-\frac{5}{4}\bm u^{n-3}+\frac{1}{5}\bm u^{n-4},\\
& B_5(\bar {\bm u}^n)=5\bar {\bm u}^n-10\bar {\bm u}^{n-1}+10\bar {\bm u}^{n-2}-5\bar {\bm u}^{n-3}+\bar {\bm u}^{n-4}.
\end{split}\end{equation}
\end{itemize}
Several remarks are in order:
\begin{itemize}

\item  We observe from \eqref{eq: NSsav3b} that $r^{n+1}$ is a first-order approximation to $E(u(\cdot,t_{n+1}))+1$ which implies that $\xi^{n+1}$ is a  first-order approximation to 1.
\item
\eqref{eq: NSsav1b} is a $k$th-order approximation to \eqref{eq:NS1b} with $k$th-order BDF for the linear terms and  $k$th-order Adams-Bashforth extrapolation for the nonlinear terms. Hence, $\bar{\bm u}^{n+1}$ is a $k$th-order approximation to $\bm u(\cdot, t^{n+1})$,
which, along with   \eqref{eq: NSsav3b} and \eqref{eq: NSsav1b}, implies that $\bm u^{n+1}$ and $p^{n+1}$ are $k$th-order approximations for $\bm u(\cdot, t^{n+1})$ and $p(\cdot, t^{n+1})$.
\item The main computational cost is to solve the Poisson type equation \eqref{eq: NSsav1b}.

\end{itemize}

\subsubsection{Fully discrete schemes with Fourier spectral method in space}
We now consider $\Omega=[0,L_x)\times [0,L_y)\times [0,L_z)$ with periodic boundary conditions.
We partition the domain $\Omega=(0, L_x) \times (0, L_y) \times (0, L_z)$ uniformly with size $h_x=L_x/N_x, h_y=L_y/N_y, h_z=L_z/N_z$ and $Nx, Ny, Nz$ are positive even integers. Then the Fourier approximation space can be defined as
\begin{equation*}\label{eq:fourierspace}
S_N = \text{span} \{e^{i\xi_jx}e^{i\eta_ky}e^{i\tau_lz}: -\frac{N_x}{2}\le j \le \frac{N_x}{2}-1,  -\frac{N_y}{2}\le k \le \frac{N_y}{2}-1,  -\frac{N_z}{2}\le l \le \frac{N_z}{2}-1\}\backslash \mathbb{R},
\end{equation*}
where $i=\sqrt{-1},\,\xi_j=2\pi j/L_x,\,\eta_k=2\pi k/L_y$ and $\tau_l=2\pi l/L_z$. Then, any function $u(x,y,z) \in L^2(\Omega)$ can be approximated by:
\begin{equation*}
u(x,y,z) \approx u_N(x,y,z)=\sum_{j=-\frac{N_x}{2}}^{\frac{Nx}{2}-1}\sum_{k=-\frac{N_y}{2}}^{\frac{N_y}{2}-1}\sum_{l=-\frac{N_z}{2}}^{\frac{N_z}{2}-1}\hat{u}_{j,k,l}e^{i\xi_jx}e^{i\eta_ky}e^{i\tau_lz},
\end{equation*}
with the Fourier coefficients defined as
\begin{equation*}
\hat{u}_{j,k,l}=\frac{1}{|\Omega|}\int_{\Omega} u e^{-i(\xi_j x+\eta_k y+\tau_l z)} d \bm x.
\end{equation*}
In the following, we fix $N_x=N_y=N_z=N$ for simplicity.

 Define the $L^2$-orthogonal projection operator $\Pi_N: L^2(\Omega) \rightarrow S_N$ by
\begin{equation*}
(\Pi_Nu-u, \Psi)=0,\quad \forall \, \Psi \in S_N,\quad u \in L^2(\Omega),
\end{equation*}
then we have the following approximation results (cf. \cite{kreiss1979stability}):
\begin{lemma}\label{Fourier}
For any $0 \le k \le m$, there exists a constant $C$ such that
\begin{equation}\label{eq:proerror}
\|\Pi_N u-u\|_k \le C\|u\|_m N^{k-m}, \forall \, u \in \bm H_p^m(\Omega).
\end{equation}
\end{lemma}
We are now ready to describe our fully discrete schemes.


Given $r^n$ and $\bm u_N^j \in S_N$ for $j=n,..., n-k+1$, we compute $\bar {\bm u}_N ^{n+1},\, r^{n+1}, p_N^{n+1}, \, \xi^{n+1}$ and $\bm u_N^{n+1}$ consecutively by
\begin{small}
\begin{subequations}\label{eq: NSsav}
\begin{align}
& \big(\frac{\alpha_k \bar{\bm u}_N^{n+1}-A_k(\bar{\bm u}_N^n)}{\delta t}, v_N \big)+\nu \big(\nabla \bar{\bm u}_N^{n+1},\nabla v_N \big)-\big(\textbf{A}(B_k({\bm u}_N^n) \cdot \nabla B_k({\bm u}_N^n)),v_N \big)=0, \quad \forall v_N \in S_N \label{eq: NSsavE1}\\
& \frac{1}{\delta t}\big(r^{n+1}-r^{n}\big)=\left\{
\begin{array}{lr}
-\nu\frac{r^{n+1}}{E(\bar{\bm u}_N^{n+1})+1} \|\Delta \bar{\bm u}_N^{n+1}\|^2,\qquad d=2,\\
-\nu\frac{r^{n+1}}{E(\bar{\bm u}_N^{n+1})+1} \|\nabla \bar{\bm u}_N^{n+1}\|^2,\qquad d=3;
\end{array}
\right.\label{eq: NSsavE2}\\
& \xi^{n+1}=\frac{r^{n+1}}{E(\bar{\bm u}_N^{n+1})+1};\label{eq: NSsavE3}\\
& \bm u_N^{n+1}=\eta_{k}^{n+1}\bar{\bm u}_N^{n+1} \;\text{ with } \eta^{n+1}_{k}=1-(1-\xi^{n+1})^{k},\label{eq: NSsavE4}
\end{align}
\end{subequations}
\end{small}
where $\alpha_k,$ the operators $A_k$ and $B_k$ $(k=1,2,3,4,5)$ are given in \eqref{eq:bdf1}-\eqref{eq:bdf5}.

Note that Fourier approximation of Poisson type equations leads to diagonal matrix in the frequency space, so the above scheme can be efficiently implemented as follows:
\begin{enumerate}[label=(\roman*)]
\item Compute $\bar {\bm u}_N^{n+1}$ from \eqref{eq: NSsavE1}, which is a Poisson-type equation;
\item  With $\bar {\bm u}_N^{n+1}$  known, determine $r^{n+1}$ explicitly  from \eqref{eq: NSsavE2};
\item Compute $\xi^{n+1}$, $\eta_k^{n+1}$ and  $\bm u_N^{n+1}$ from \eqref{eq: NSsavE3} and \eqref{eq: NSsavE4}, goto the next step.
\end{enumerate}
Finally, whenever pressure is needed, it can be computed from
\begin{equation}\label{prePoissonN}
 \Delta p_N^{n+1}=-\Pi_N \nabla \cdot (\bm u_N^{n+1} \cdot \nabla \bm u_N^{n+1}).
\end{equation}

\subsection{Stability results}
We have the following results concerning the stability of the above schemes.
\begin{theorem}\label{stableThm} 
Let $\bm u_0\in \bm V\cap \bm H^2_p$ if $d=2$ and $\bm u_0\in \bm V$ if $d=3$.
Let $\{r^{k},\,\xi^{k},\,\bar{\bm u}_N^k,\,{\bm u}_N^k\}$ be the solution of  the fully discrete  scheme  \eqref{eq: NSsav}. Then,
given $r^n \ge 0$, we have $r^{n+1} \ge 0$, $\xi^{n+1} \ge 0$,  and for any $k$, the  scheme \eqref{eq: NSsav} is unconditionally energy stable in the sense that
\begin{equation}\label{eq: energystable}
r^{n+1}-r^{n}=\left\{
\begin{array}{lr}
-\delta t \nu \xi^{n+1}\|\Delta \bar{\bm u}_N^{n+1}\|^2 \le 0,\qquad d=2,\\
-\delta t \nu \xi^{n+1}\|\nabla \bar{\bm u}_N^{n+1}\|^2 \le 0,\qquad d=3,
\end{array}
\right. \,\quad \forall n.
\end{equation}
Furthermore, there exists $M_k>0$ such that
\begin{equation}\label{eq: L2bound}
\begin{array}{lr}
\|\nabla {\bm u}_N^{n+1}\|^2\le M^2_k,\qquad d=2,\\
\|{\bm u}^{n+1}_N\|^2 \le M^2_k,\qquad d=3,
\end{array}
\,\quad \forall n.
\end{equation}
Same results hold for the semi-discrete schemes \eqref{eq: NSsavb} with $\bar{\bm u}_N^{n+1}$ and ${\bm u}_N^{n+1}$ in \eqref{eq: energystable} and \eqref{eq: L2bound} be replaced by $\bar{\bm u}^{n+1}$ and ${\bm u}^{n+1}$.
\end{theorem}
\begin{proof}
Since the proofs for the fully discrete scheme \eqref{eq: NSsav} and for the semi-discrete scheme \eqref{eq: NSsavb} are essentially the same, we shall only give the proof for the fully discrete scheme \eqref{eq: NSsav} below.

Given $r^n \ge 0$. Since $E(\bar{\bm u}_N^{n+1}) >0$, it follows from \eqref{eq: NSsavE2} that
\begin{equation*}
r^{n+1}=\left\{
\begin{array}{lr}
\frac{r^{n}}{1+\delta t \nu\frac{\|\Delta \bar{\bm u}_N^{n+1}\|^2 }{E(\bar{\bm u}_N^{n+1})+1}} \ge 0,\qquad d=2,\\
\frac{r^{n}}{1+\delta t \nu\frac{\|\nabla \bar{\bm u}_N^{n+1}\|^2 }{E(\bar{\bm u}_N^{n+1})+1}} \ge 0,\qquad d=3.
\end{array}
\right.
\end{equation*}
Then we derive from \eqref{eq: NSsavE3} that $\xi^{n+1} \ge 0$ and obtain \eqref{eq: energystable}.

Denote $M:=r^0=E[\bm u(\cdot,0)]$, then \eqref{eq: energystable} implies $r^n \le M,\, \forall n$. It then follows from \eqref{eq: NSsavE3}  that
\begin{equation}\label{eq: xibound}
|\xi^{n+1}|=\frac{r^{n+1}}{E(\bar{\bm u}_N^{n+1})+1} \le \left\{
\begin{array}{lr}
\frac{2M}{\|\nabla \bar{\bm u}_N^{n+1}\|^2+2},\qquad d=2,\\
\frac{2M}{\|\bar{\bm u}_N^{n+1}\|^2+2},\qquad d=3.
\end{array}
\right.
\end{equation}
Since $\eta_k^{n+1}=1-(1-\xi^{n+1})^{k}$, we have $\eta_k^{n+1}=\xi^{n+1}P_{k-1}(\xi^{n+1})$ with   $P_{k-1}$ being a polynomial  of degree $k-1$. Then, we derive from \eqref{eq: xibound} that  there exists $M_k>0$ such that
\begin{equation*}\label{eq: etabound}
|\eta_k^{n+1}|= |\xi^{n+1}P_{k-1}(\xi^{n+1})| \le \left\{
\begin{array}{lr}
\frac{{M_k}}{\|\nabla \bar{\bm u}_N^{n+1}\|^2+2},\qquad d=2,\\
\frac{{M_k}}{\|\bar{\bm u}_N^{n+1}\|^2+2},\qquad d=3,
\end{array}
\right.
\end{equation*}
which, along with  $\bm u_N^{n+1}=\eta_k^{n+1}\bar{\bm u}_N^{n+1}$, implies
\begin{equation*}
\begin{array}{lr}
\|\nabla {\bm u}_N^{n+1}\|^2=(\eta_k^{n+1})^2\|\nabla \bar{\bm u}_N^{n+1}\|^2 \le \big(\frac{{M_k}}{\|\nabla \bar{\bm u}_N^{n+1}\|^2+2}\big)^2\|\nabla \bar{\bm u}_N^{n+1}\|^2 \le M^2_k,\qquad d=2,\\
\|{\bm u}_N^{n+1}\|^2=(\eta_k^{n+1})^2\|\bar{\bm u}_N^{n+1}\|^2 \le \big(\frac{{M_k}}{\|\bar{\bm u}_N^{n+1}\|^2+2}\big)^2\|\bar{\bm u}_N^{n+1}\|^2 \le M^2_k,\qquad d=3.
\end{array}
\end{equation*}
\end{proof}

\subsection{Numerical examples}
Before we start the error analysis,  we provide numerical examples to demonstrate the convergence rates and compare the performance of the  schemes with different orders on a classical benchmark problem.

\textit{Example 1: Convergence test.} Consider the
Navier-Stokes equations \eqref{eq:NS} with an external forcing ${\bm f}$ in $\Omega=(0,2)\times (0,2)$ with periodic boundary condition
such that the  exact solution is given by
\begin{align*}
& u_1(x,y)=\pi \exp(\sin(\pi x))\exp(\sin(\pi y))\cos(\pi y) \sin^2(t);\\
& u_2(x,y)=-\pi \exp(\sin(\pi x))\exp(\sin(\pi y))\cos(\pi x) \sin^2(t);\\
& p(x,y)=\exp(\cos(\pi x)\sin(\pi y))\sin^2(t).
\end{align*}
We set $\nu=1$ in \eqref{eq:NS}, and use the Fourier spectral method with $40\times 40$ modes for space discretization so that the spatial discretization error is negligible with respect to the time discretization error.
 In Figures \ref{fig: NStest1}, we  plot the  convergence rate of the $H^1$ error for the velocity and the pressure at $T=1$  by using first- to fourth-order schemes. We  observe  the expected convergence rates  for both the velocity and the pressure.

\begin{figure}[htbp]
\begin{center}
\subfigure[BDF1 errors of velocity and pressure]{ \includegraphics[scale=.35]{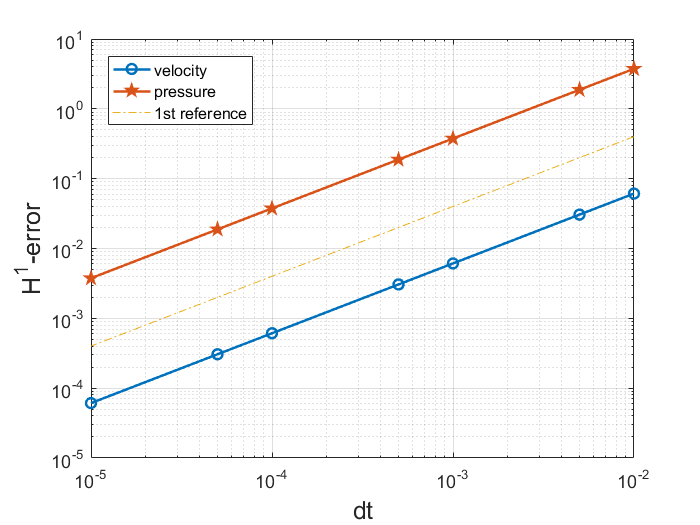}}
  \subfigure[BDF2 errors of velocity and pressure]{ \includegraphics[scale=.35]{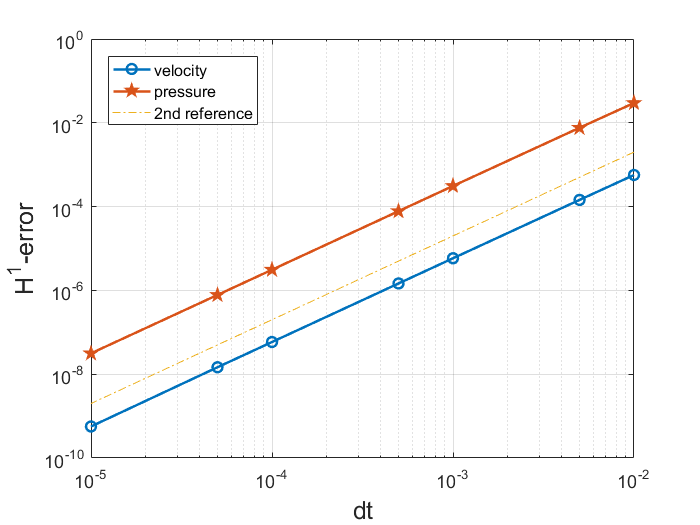}}\\
    \subfigure[BDF3 errors of velocity and pressure]{ \includegraphics[scale=.35]{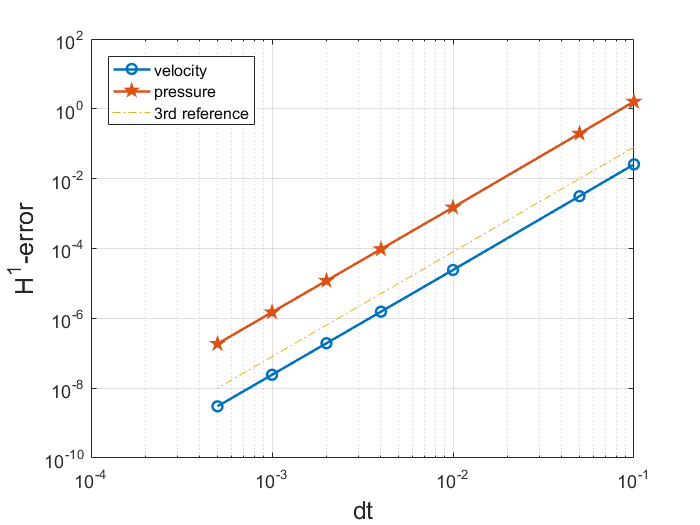}}
      \subfigure[BDF4 errors of velocity and pressure ]{ \includegraphics[scale=.35]{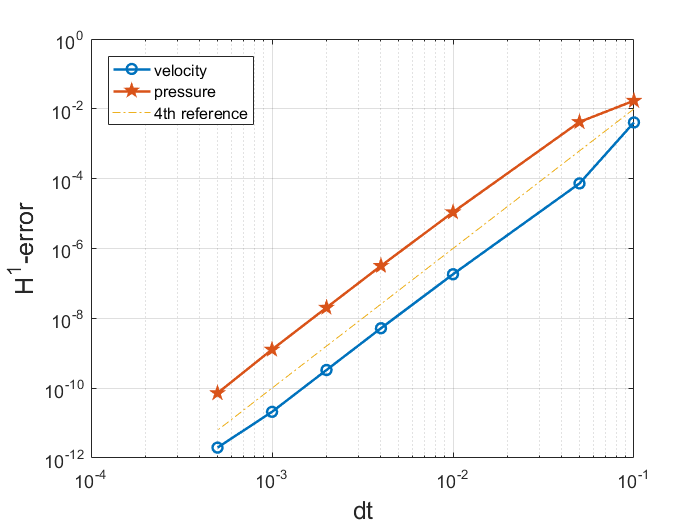}}
  \caption{\small Convergence test for the Navier-stokes equations using SAV/BDF$k$ $(k=1,2,3,4)$ }
  \label{fig: NStest1}
\end{center}
\end{figure}
\textit{Example 2: Double shear layer problem \cite{bell1989second,brown1995,di2005moving}.}  Consider  the Navier-Stokes equations \eqref{eq:NS} in $\Omega=(0,1)\times (0,1)$ with periodic boundary conditions and the initial condition  given by
\begin{align*}
&u_1(x,y,0)= \left \{
\begin{aligned}
&\tanh(\rho (y-0.25)),\,\,  y\le 0.5\\
&\tanh(\rho (0.75-y)),\,\,  y>0.5
\end{aligned}
\right. ,\\
&u_2(x,y,0)= \delta \sin(2 \pi x),
\end{align*}
where $\rho$ determines the slope of the shear layer and $\delta$ represents the size of the perturbation. In our simulations, we fix $\delta=0.05$.
\begin{figure}[htbp]
\begin{center}
  \subfigure[ 1st order]{ \includegraphics[scale=.35]{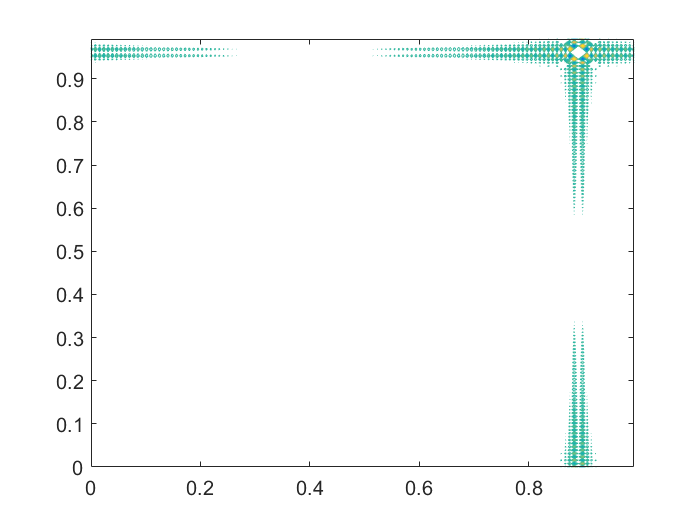}}
    \subfigure[2nd-order ]{ \includegraphics[scale=.35]{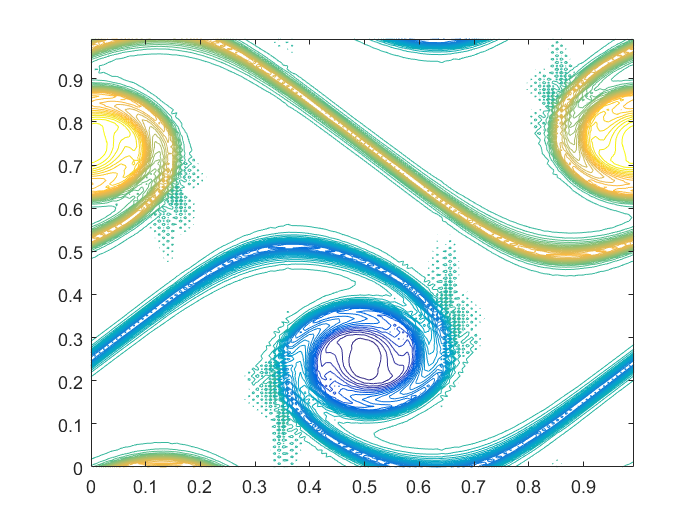}}\\
      \subfigure[3rd-order ]{ \includegraphics[scale=.35]{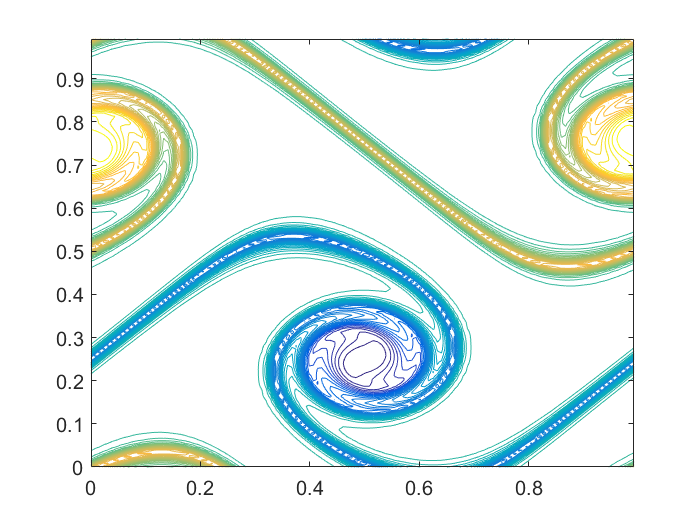}}
      \subfigure[4th-order  ]{ \includegraphics[scale=.35]{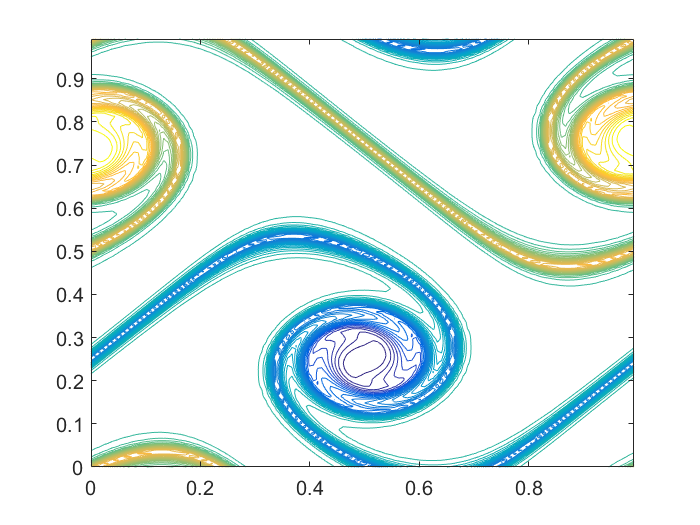}}
  \caption{\small  Thick layer problem: vorticity contours at T=1.2 with $\rho=30$, $\nu=0.0001$ and $\delta t=8 \times 10^{-4}$ }
  \label{fig: Thick}
\end{center}
\end{figure}
We first test a \textit{thick layer problem} by choosing $\rho=30$ and $\nu=0.0001$. We use the Fourier spectral method with $128\times 128$ modes for the space discretization, and set $\delta t=8 \times 10^{-4}$. In Figures \ref{fig: Thick}, we show the vorticity contours at $T=1.2$ obtained with first- to fourth-order schemes.  We observe that correct solution is obtained with the third- and fourth-order schemes while the first-order scheme gives totally wrong result  and the second-order scheme leads to inaccurate result.  

\begin{figure}[htbp]
\begin{center}
  \subfigure[1st-order]{ \includegraphics[scale=.35]{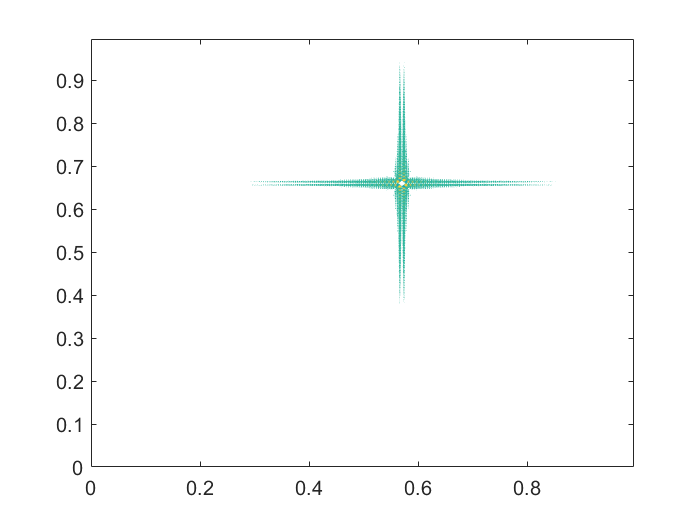}}
    \subfigure[ 2nd-order]{ \includegraphics[scale=.35]{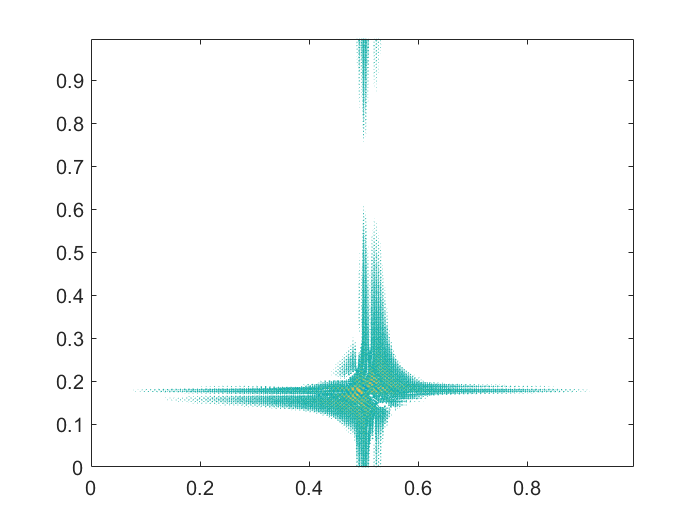}}\\
      \subfigure[3rd-order ]{ \includegraphics[scale=.35]{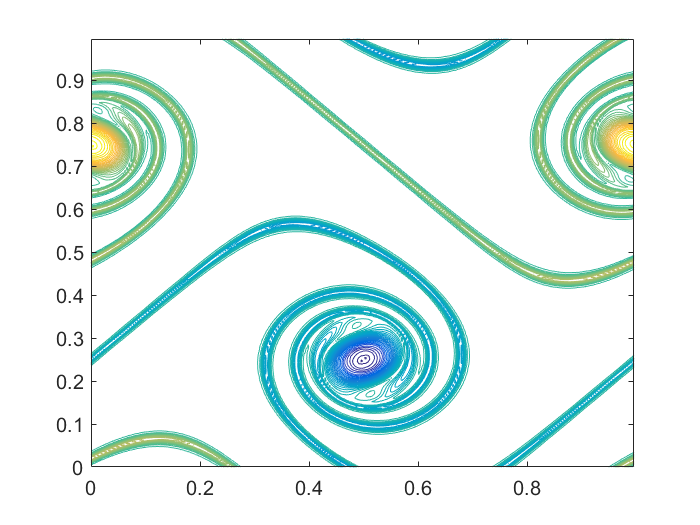}}
      \subfigure[4th-order  ]{ \includegraphics[scale=.35]{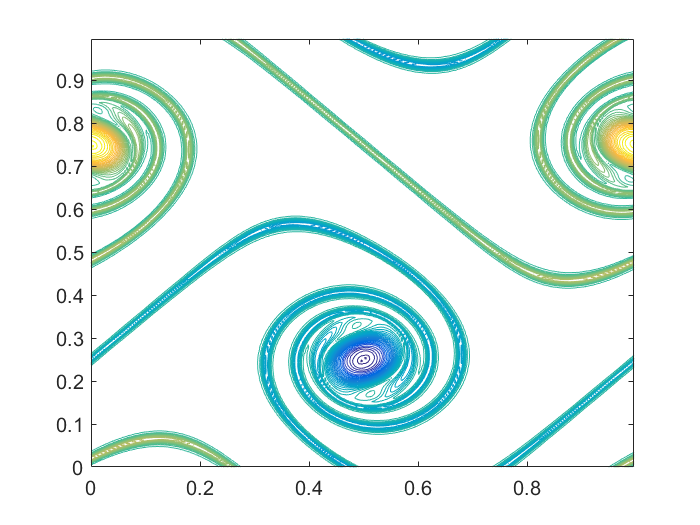}}
  \caption{\small  Thin layer problem: vorticity contours at T=1.2 with $\rho=100$, $\nu=0.00005$ and $\delta t=3 \times 10^{-4}$}
  \label{fig: Thin}
\end{center}
\end{figure}

Next, we test a \textit{thin layer problem} by choosing $\rho=100$ and $\nu=0.00005$. We use first- to the fourth-order schemes with $256\times 256$ Fourier modes and $\delta t=3 \times 10^{-4}$. In Figures \ref{fig: Thin}, we plot the vorticity contours at $T=1.2$. We observe that correct solutions are obtained with the third- and fourth-order schemes while first- and second-order schemes lead to wrong results.

In order to examine the effect of SAV approach, we
plot in Figure \ref{fig: Thin2} evolution of the SAV factor $\eta=1-(1-\xi)^2$ and  the vorticity contours at $T=1.2$, computed with the second-order scheme with $\delta t=2.5 \times 10^{-4}$.  We observe that  at around $t=1.05$, where the usual semi-implicit second-order scheme blows up, the SAV factor dips slightly to allow the scheme  continue to produce correct simulation.

\begin{rem}
 These two tests indicate that for high Reynolds number flows with complex structures, higher-order schemes are preferred over lower-order schemes, as much smaller time steps have to be used to obtain correct solutions with lower-order schemes.

Note that if we use the usual semi-implicit schemes with the same time steps in the above tests, the first- and second-order schemes would blow up. So the SAV approach can effectively prevent the numerical solution from blowing up although sufficient small time steps are needed to capture the correct solution. Thus, one is advised to adopt a suitable adaptive time stepping  to take  full advantage of the SAV schemes.
\end{rem}

\begin{figure}[htbp]
\begin{center}
    \subfigure[Evolution of $\eta$ ]{ \includegraphics[scale=.35]{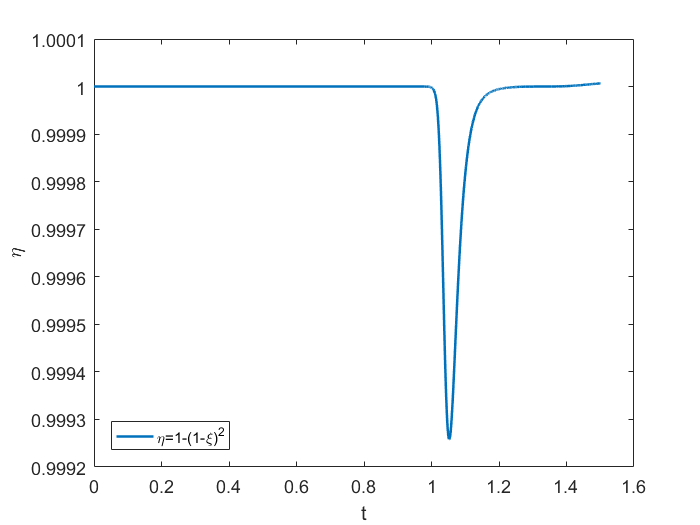}}
     \subfigure[Vorticity contours at $t=1.2$]{ \includegraphics[scale=.35]{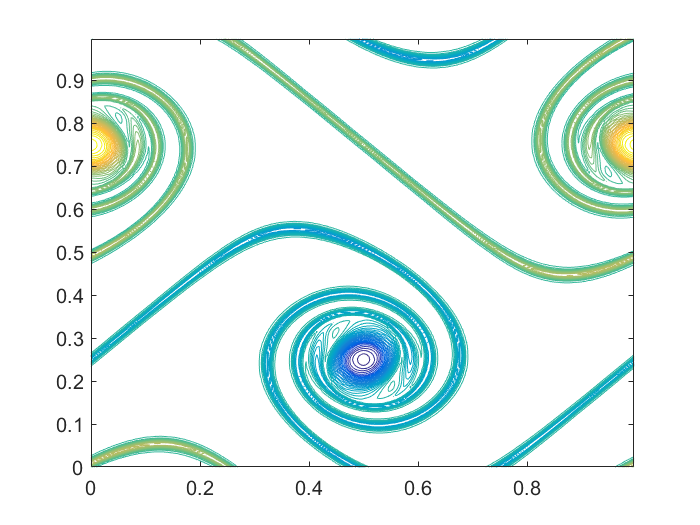}}
  \caption{\small  Thin layer problem: second-order scheme with $\rho=100$, $\nu=0.00005$ and $\delta t=2.5 \times 10^{-4}$}
  \label{fig: Thin2}
\end{center}
\end{figure}

\section{Error analysis}
In this section, we carry out a unified error analysis for the fully discrete schemes   \eqref{eq: NSsav} with $1\le k\le5$, and state, as  corollaries,  similar results for the semi-discrete schemes  \eqref{eq: NSsavb}.

We denote
\begin{equation*}
\begin{split}
& t^n=n\,\delta t,\quad s^n=r^n-r(t^n),\\
& \bar{\bm e}_N^n=\bar{\bm u}_N^{n}- \Pi_N \bm u(\cdot,t^n),\quad  \bm e_N^n={\bm u}_N^{n}- \Pi_N \bm u(\cdot,t^n),\quad {\bm e}_\Pi^n=\Pi_N \bm u(\cdot,t^n) -\bm u(\cdot,t^n),\\
& \bar{\bm e}^n=\bar{\bm u}_N^{n}-\bm u(\cdot,t^n)=\bar{\bm e}_N^n+{\bm e}_\Pi^n,\quad
 \bm e^n=\bm u_N^n-\bm u(\cdot,t^n)=\bm e_N^n+{\bm e}_\Pi^n.
\end{split}
\end{equation*}
To simplify the notations, we  dropped the dependence on $N$ for $\bar{\bm e}^n$ and ${\bm e}^n$ in the above, and will do so for some other quantities in the sequel.

\subsection{Several useful lemmas}
We will frequently use the following two discrete versions of the Gronwall lemma.
\begin{lemma}\label{Gron1}
\textbf{(Discrete Gronwall Lemma 1 \cite{STWbook})} Let $y^k,\,h^k,\,g^k,\,f^k$ be four nonnegative sequences satisfying
\begin{equation*}
y^n+\delta t \sum_{k=0}^{n}h^k \le B+\delta t \sum_{k=0}^{n}(g^ky^k+f^k)\; \text{ with }\; \delta t \sum_{k=0}^{T/\delta t} g^k \le M,\, \forall\, 0\le n \le T/\delta t.
\end{equation*}
We assume $\delta t\, g^k <1$ for all $k$, and let $\sigma=\max_{0\le k \le T/\delta t}(1-\delta t g^k)^{-1}$. Then
\begin{equation*}
y^n+\delta t \sum_{k=1}^{n}h^k \le \exp(\sigma M)(B+\delta t\sum_{k=0}^{n}f^k),\,\,\forall\, n\le T/\delta t.
\end{equation*}
\end{lemma}

\begin{lemma}\label{Gron2}
\textbf{(Discrete Gronwall Lemma 2 \cite{shen1990long})} Let $,a_n,\,b_n,\,c_n,$ and $d_n$ be four nonnegative sequences satisfying \begin{equation*}
a_m+\tau \sum_{n=1}^{m} b_n \le \tau \sum_{n=0}^{m-1}a_n d_n +\tau \sum_{n=0}^{m-1} c_n+ C, \,m \ge 1,
\end{equation*}
where $C$ and $\tau$ are two positive constants.
Then
\begin{equation*}
a_m+\tau \sum_{n=1}^{m} b_n \le \exp\big(\tau \sum_{n=0}^{m-1} d_n \big)\big(\tau \sum_{n=0}^{m-1}c_n+C \big),\,m \ge 1.
\end{equation*}
\end{lemma}

 Based on Dahlquist's G-stability theory,
Nevanlinna and Odeh  \cite{nevanlinna1981} proved the following result  which  plays an essential role in our error analysis.

\begin{lemma}\label{lemmaODEH}
 For $1\le k\le 5$, there exist $0 \le \tau_k < 1$, a positive definite symmetric matrix $G=(g_{ij}) \in \mathcal{R}^{k,k}$ and real numbers $\delta_0,..., \delta_k$ such that
\begin{equation*}
\begin{split}
\Big(\alpha_k u^{n+1}-A_k(u^n), u^{n+1}-\tau_k u^n \Big)&=\sum_{i,j=1}^{k}g_{ij}(u^{n+1+i-k},u^{n+1+j-k})\\
&-\sum_{i,j=1}^{k}g_{ij}(u^{n+i-k},u^{n+j-k})
+\|\sum_{i=0}^{k} \delta_i u^{n+1+i-k}\|^2,
\end{split}
\end{equation*}
where the smallest possible values of $\tau_k$ are
\begin{equation*}
\tau_1=\tau_2=0,\quad \tau_3=0.0836,\quad \tau_4=0.2878,\quad \tau_5=0.8160,
\end{equation*}
and $\alpha_k$, $A_k$ are defined in \eqref{eq:bdf3}-\eqref{eq:bdf5}.
\end{lemma}

We also recall  the following lemma \cite{liu2010stable} which will be used to prove local error estimates in the three-dimensional case.

\begin{lemma}\label{locallemma}
Let $\phi:(0,\infty)\rightarrow (0,\infty)$ be continuous and increasing, and let $M>0$. Given $T_*$ such that $0< T_*< \int_M^{\infty}dz/\phi(z)$, there exists $C_*>0$ independent of $\delta t>0$ with the following property. Suppose that quantities $z_n,\,w_n \ge 0$ satisfy
\begin{equation*}
z_n+\sum_{k=0}^{n-1}\delta t w_k \le y_n:=M+\sum_{k=0}^{n-1} \delta t \phi(z_k), \,\forall n \le n_*.
\end{equation*}
with $n_* \delta t\le T_*$. Then $y_{n_*} \le C_*$.
\end{lemma}

\subsection{Error analysis  for the velocity  in 2D}

\begin{theorem}\label{ThmNS}
Let $d=2$, $T>0$, $\bm u_0\in \bm V\cap \bm H_p^m$ with $m\ge 3$ and $\bm u$ be the solution of \eqref{eq:NS}.
We  assume that $\bar {\bm u}_N^i$ and $\bm u_N^i$ $(i=1,\cdots,k-1)$ are computed with a proper initialization procedure such that
\begin{equation}\label{inik}
\begin{split}
&\|\bar {\bm u}_N^i-\bm u(\cdot,t_i)\|_1,\;\| \bm u_N^i-\bm u(t_i)\|_1=O(\delta t^k+N^{1-m}),\\
&\|\bar {\bm u}_N^i-\bm u(\cdot,t_i)\|_2,\;\| \bm u_N^i-\bm u(t_i)\|_2=O(\delta t^k+N^{2-m}),
\end{split} \quad i=1,2,3,4,5.
\end{equation}
Let $\bar {\bm u}_N^{n+1}$ and $\bm u_N^{n+1}$ be computed with the $k$th-order scheme  \eqref{eq: NSsav} $(1 \le k \le5)$, and
\begin{equation*}
 \eta_1^{n+1}=1-(1-\xi^{n+1})^{2}, \quad \eta_k^{n+1}=1-(1-\xi^{n+1})^{k}\; (k=2,3,4,5).
\end{equation*}
Then for  $n+1 \le T/\delta t$ with $\delta t \le \frac1{1+2^{k+2} C_0^{k+1}}$ and $N \ge {2^{k+2}C_{\Pi}^{k+1}+1} $, we have
\begin{equation*}
\|\bar{\bm u}_N^{n}-\bm u(\cdot,t^n)\|_1^2,\,\| {\bm u}_N^{n}-\bm u(\cdot,t^n)\|_1^2 \le C \delta t^{2k}+C N^{2(1-m)},
\end{equation*}
and
\begin{equation*}
\delta t\sum_{q=0}^n \|\bar {\bm u}_N^{q+1}-\bm u(\cdot,t^{q+1})\|_2^2,\,\delta t\sum_{q=0}^n \|{\bm u}_N^{q+1}-\bm u(\cdot,t^{q+1})\|_2^2 \le C \delta t^{2k}+C N^{2(2-m)}.
\end{equation*}
where the constants $C_0$, $C_{\Pi}$ and $C$ are dependent on $T,\, \Omega,$ the $k\times k$ matrix $G=(g_{ij})$ in Lemma \ref{lemmaODEH} and the exact solution $\bm u$, but are independent of $\delta t$ and $N$.


\end{theorem}
\begin{proof}  
It is shown in \cite{Tema83} that in the periodic case, $\bm u_0\in \bm H_p^m$ implies that $\bm u(\cdot,t)\in \bm H_p^m$ for all $t\le T$, and furthermore, it is shown in \cite{foias1989gevrey} that  $\bm u$ has Gevrey class regularity. In particular, we have
\begin{equation}\label{regNS}
\bm u\, \in C([0,T];\bm H_p^m),\, m \ge 3,\, \frac{\partial ^j \bm u}{\partial t^j} \in L^2(0,T;\bm H_p^2) \,\, 1\le j \le k,\;
\frac{\partial ^{k+1} \bm u}{\partial t^{k+1}} \in L^2(0,T;L_0^2).
\end{equation}
To simplify the presentation,  we assume
 $\bar{\bm u}_N^i=\bm u_N ^i=\Pi_N \bm u(t_i)$ and $r^i=E_1[\bm u_N^i]$ for $i=1,\cdots,k-1$ so that \eqref{inik} is obviously satisfied.

 The main task  is to prove by induction,
\begin{equation}\label{eq: prestepsk}
|1-\xi^q| \le C_0\,\delta t+C_{\Pi} N^{2-m},\,\, \forall  q\le T/{\delta t},
\end{equation}
where  the constant $C_0$ and $C_{\Pi} $   will be defined in the induction process below.

Under the assumption, \eqref{eq: prestepsk}  certainly holds for $q=0$. Now
suppose we have
\begin{equation}\label{eq: prestepsk2}
|1-\xi^q| \le C_0\,\delta t+C_{\Pi} N^{2-m},\,\, \forall q \le n,
\end{equation}
we shall prove below
\begin{equation}\label{eq: xik}
|1-\xi^{n+1}| \le C_0 \delta t+C_{\Pi} N^{2-m}.
\end{equation}
We shall first consider $k=2,3,4,5$, and point out the necessary modifications for the case $k=1$ later.

\textbf{Step 1: Bounds for $\nabla \bar{\bm u}_N^q$, $\Delta \bar{\bm u}_N^q$ and  $\Delta {\bm u}_N^q$,  $\forall q\le n$.} We first recall the inequality
\begin{equation}\label{eq:inequk}
(a+b)^k \le 2^k (a^k+b^k),\quad \forall a,b>0,\,k\ge 1.
\end{equation}
Under the assumption \eqref{eq: prestepsk2}, if we choose $\delta t$ small enough and $N$ large enough such that
\begin{equation}\label{eq: dtcond1}
 \delta t \le \text{min}\{\frac{1}{2^{k+2} C_0^{k}},1\} ,\quad N \ge  \text{max}\{2^{k+2} C_{\Pi}^k,1\},
\end{equation}
 we have
\begin{equation}\label{eq: erroretaHigh}
1-(\frac{1}{2^{k+2}C_0^{k-1}}+\frac{N^{3-m}}{2^{k+2}C_{\Pi}^{k-1}})\le |\xi^q|\le 1+(\frac{1}{2^{k+2}C_0^{k-1}}+\frac{N^{3-m}}{2^{k+2}C_{\Pi}^{k-1}}),\,\, \forall q\le n,
\end{equation}
and
\begin{equation*}
(1-\xi^q)^k \le \frac{\delta t^{k-1}}{4}+\frac{N^{k(2-m)+1}}{4},\,\,\forall q\le n,
\end{equation*}
and
\begin{equation*}
\frac{1}{2}<1-(\frac{\delta t^{k-1}}{4}+\frac{N^{k(2-m)+1}}{4})\le |\eta_k^q| \le 1+\frac{\delta t^{k-1}}{4}+\frac{N^{k(2-m)+1}}{4}<2,\,\, \forall q\le n.
\end{equation*}
Then it follows from the above and \eqref{eq: L2bound} that
\begin{equation}\label{eq:uL2bound}
\|\bar{\bm u}_N^{q}\|_1 \le 2M_k,\,\forall q\le n.
\end{equation}
Moveover, \eqref{eq: energystable} and $m\ge 3$ imply that
\begin{equation}\label{eq:H1sumbound}
\nu \delta t \sum_{q=1}^{n}\|\Delta \bar{\bm u}_N^q\|^2 \le  \frac{2r^0}{|\xi^q|} \le 4r^0,\, \,C_0 \ge 1,\,\, C_{\Pi} \ge 1.
\end{equation}
and
\begin{equation}\label{eq:H2sumbound}
\nu \delta t \sum_{q=1}^{n}\|\Delta {\bm u}_N^q\|^2  \le 16r^0,\, \,C_0 \ge 1,\,\, C_{\Pi} \ge 1.
\end{equation}
\textbf{Step 2: Estimates for $\nabla \bar{\bm e}_N^{n+1}$ and $\Delta \bar{\bm e}_N^{n+1}$.} By the assumptions on the exact solution $\bm u$ and \eqref{eq:uL2bound}, we can choose $C$ large enough such that
\begin{equation}\label{eq:H1boundex}
\|\bm u(t)\|_{H^2}^2 \le C,\,\forall t\le T,\,\,\|\bar {\bm u}_N^{q}\|_1 \le C,\,\,\forall q\le n.
\end{equation}
From \eqref{eq: NSsavE1}, we can write down the error equation as
\begin{equation}\label{eq:error}
\big(\alpha_k\bar{\bm e}^{q+1}-A_k(\bar{\bm e}^q), v_N\big)+\delta t \nu \big(\nabla \bar{\bm e}^{q+1}, \nabla v_N \big)= \big(R^q_k,v_N \big) +\delta t \big(Q^q_k, v_N\big),\quad \forall v_N \in S_N,
\end{equation}
where  $Q^q_k$  and $R^q_k$ are given by
\begin{equation}\label{eq: R2high}
Q^q_k=-\textbf{A}\big((B_k{\bm u}^q) \cdot \nabla) B_k({\bm u}^q)\big)+\textbf{A}\big({\bm u}(t^{q+1}) \cdot \nabla {\bm u}(t^{q+1})\big),
\end{equation}
and
\begin{equation}\label{eq: R1high}
\begin{split}
R^q_k & =-\alpha_k \bm u(t^{q+1})+A_k(\bm u(t^q))+\delta t \bm u_t(t^{q+1}) \\
& =\sum_{i=1}^{k} a_i \int_{t^{q+1-i}}^{t^{q+1}}(t^{q+1-i}-s)^k\frac{\partial^{k+1} \bm u}{\partial t^{k+1}}(s)ds,
\end{split}
\end{equation}
with $a_i$ being some fixed and bounded constants determined by the truncation errors,
 for example, in the case $k=3$, we have
\begin{equation*}
R_3^q=-3\int_{t^q}^{t^{q+1}}(t^q-s)^3 \frac{\partial^{4} \bm u}{\partial t^{4}}(s)ds+\frac{3}{2}\int_{t^{q-1}}^{t^{n+1}}(t^{q-1}-s)^3 \frac{\partial^{4} \bm u}{\partial t^{4}}(s)ds-\frac{1}{3}\int_{t^{q-2}}^{t^{n+1}}(t^{q-2}-s)^3 \frac{\partial^{4} \bm u}{\partial t^{4}}(s)ds.
\end{equation*}

Let $v_N=-\Delta \bar{\bm e}_N^{q+1}+\tau_k \Delta \bar{\bm e}_N^{q}$ in \eqref{eq:error}, it follows from Lemma \ref{lemmaODEH} and \eqref{eq:proerror} that
\begin{equation}\label{eq: errorhigh}
\begin{split}
\sum_{i,j=1}^{k}g_{ij}&( \nabla \bar{\bm e}_N^{q+1+i-k}, \nabla \bar{\bm e}_N^{q+1+j-k}) -\sum_{i,j=1}^{k}g_{ij}( \nabla \bar{\bm e}_N^{q+i-k}, \nabla \bar{\bm e}_N^{q+j-k})\\
&+\|\sum_{i=0}^{k} \delta_i \nabla \bar{\bm e}_N^{q+1+i-k}\|^2
+\delta t \nu \|\Delta \bar{\bm e}_N^{q+1}\|^2\\
& = \delta t \nu (\Delta \bar{\bm e}_N^{q+1}, \tau_k \Delta \bar{\bm e}_N^{q})+ (R^q_k, -\Delta \bar{\bm e}_N^{q+1}+\tau_k \Delta \bar{\bm e}_N^{q})+\delta t (Q^n_k, -\Delta \bar{\bm e}_N^{q+1}+\tau_k \Delta \bar{\bm e}_N^{q}).
\end{split}
\end{equation}
Next, we bound the righthand side of \eqref{eq: errorhigh}.
It follows from \eqref{eq: R1high} that
\begin{equation}\label{eq: absR1high}
\|R^q_k\|^2 \le C \delta t^{2k+1} \int_{t^{q+1-k}}^{t^{q+1}} \|\frac{\partial^{k+1} \bm u}{\partial t^{k+1}}(s)\|^2 ds.
\end{equation}
Therefore,
\begin{equation}\label{eq: errorR1high}
\begin{split}
\Big |\big(R^q_k, &-\Delta \bar{\bm e}_N^{q+1}+\tau_k \Delta \bar{\bm e}_N^{q}\big)\Big | \le \frac{C(\epsilon)}{\delta t}\| R^q_k\|^2+ \delta t \epsilon \|-\Delta \bar{\bm e}_N^{q+1}+\tau_k \Delta \bar{\bm e}_N^{q}\|^2 , \\
& \le \frac{C(\epsilon)}{\delta t}\|R^q_k\|^2 +2\delta t\epsilon\|\Delta \bar{\bm e}_N^{q+1}\|^2+2\delta t\epsilon\|\Delta \bar{\bm e}_N^{q}\|^2 , \\
& \le 2\delta t\epsilon\|\Delta \bar{\bm e}_N^{q+1}\|^2+2\delta t\epsilon\|\Delta \bar{\bm e}_N^{q}\|^2 +C(\epsilon) \delta t^{2k} \int_{t^{q+1-k}}^{t^{q+1}} \|\frac{\partial^{k+1}\bm u}{\partial t^{k+1}}(s)\|^2 ds.
\end{split}
\end{equation}
For the term with $Q_k^q$, we split it as
\begin{equation}\label{eq:R2split}
\begin{split}
(Q_k^n, -\Delta \bar{\bm e}_N^{q+1}+\tau_k \Delta {\bm e}_N^{q})&=\Big(\textbf{A}\big([\bm u(t^{q+1})-B_k({\bm u}^q)] \cdot \nabla \bm u(t^{q+1})\big), -\Delta \bar{\bm e}_N^{q+1}+\tau_k \Delta {\bm e}_N^{q} \Big)\\
& +\Big(\textbf{A}\big(B_k({\bm u}^q) \cdot \nabla[\bm u(t^{q+1})-B_k(\bm u(t^q))]\big),  -\Delta \bar{\bm e}_N^{q+1}+\tau_k \Delta \bar{\bm e}_N^{q} \Big)\\
& -\Big(\textbf{A}\big( B_k({\bm e}^q) \cdot \nabla B_k({\bm e}^q) \big),  -\Delta \bar{\bm e}_N^{q+1}+\tau_k \Delta \bar{\bm e}_N^{q} \Big) \\
& - \Big(\textbf{A}\big(B_k(\bm u(t^q)) \cdot \nabla B_k({\bm e}^q) \big),  -\Delta \bar{\bm e}_N^{q+1}+\tau_k \Delta \bar{\bm e}_N^{q} \Big).
\end{split}
\end{equation}
We bound the terms on the right hand side of \eqref{eq:R2split} with the help of \eqref{eq:ineq2d}, \eqref{eq:ineq2} and \eqref{eq:H1boundex}:
\begin{equation}\label{eq:non1}
\begin{split}
&\Big(\textbf{A}\big([\bm u(t^{q+1}) -B_k({\bm u}^q)] \cdot \nabla \bm u(t^{q+1})\big), -\Delta \bar{\bm e}_N^{q+1}+\tau_k \Delta \bar{\bm e}_N^{q} \Big)  \\
& \le C\|\bm u(t^{q+1})-B_k({\bm u}^q)\|_1 \|\bm u(t^{q+1})\|_2 \|-\Delta \bar{\bm e}_N^{q+1}+\tau_k \Delta \bar{\bm e}_N^{q}\|\\
& \le C(\epsilon)\|\bm u(t^{q+1})-B_k({\bm u}^q)\|_1^2 \|\bm u(t^{q+1})\|_2^2+ \epsilon \|-\Delta \bar{\bm e}_N^{q+1}+\tau_k \Delta \bar{\bm e}_N^{q}\|^2\\
& \le C(\epsilon)\|\bm u(t^{q+1})-B_k(\bm u(t^{q}))\|_1^2\|\bm u(t^{q+1})\|_2^2+C(\epsilon)\|B_k({\bm e}^q)\|_1^2\|\bm u(t^{q+1})\|_2^2+\epsilon\|-\Delta \bar{\bm e}_N^{q+1}+\tau_k \Delta \bar{\bm e}_N^{q}\|^2\\
& \le C(\epsilon)\|\sum_{i=1}^{k}b_i \int_{t^{q+1-i}}^{t^{q+1}}(t^{q+1-i}-s)^{k-1}\frac{\partial^{k} \bm u}{\partial t^{k}}(s)ds\|_1^2 + C(\epsilon)\|B_k({\bm e}^q)\|_1^2+2\epsilon\|\Delta \bar{\bm e}_N^{q+1}\|^2+2\epsilon\|\Delta \bar{\bm e}_N^{q}\|^2\\
& \le C(\epsilon)\delta t^{2k-1} \int_{t^{q+1-k}}^{t^{q+1}}\|\frac{\partial^{k} \bm u}{\partial t^{k}}(s)\|_1^2ds + C(\epsilon)\|B_k({\bm e}^q)\|_1^2+2\epsilon\|\Delta \bar{\bm e}_N^{q+1}\|^2+2\epsilon\|\Delta \bar{\bm e}_N^{q}\|^2,
\end{split}
\end{equation}
where $b_i$ are some fixed and bounded constants determined by the truncation error. For example, in the case $k=3$, we have
\begin{equation*}
\begin{split}
B_3(\bm u(t^q))-\bm u(t^{q+1})& =-\frac{3}{2}\int_{t^{q}}^{t^{q+1}}(t^{q}-s)^2\frac{\partial^{3} \bm u}{\partial t^{3}}(s)ds+\frac{3}{2}\int_{t^{q-1}}^{t^{q+1}}(t^{q-1}-s)^2\frac{\partial^{3} \bm u}{\partial t^{3}}ds \\
& -\frac{1}{2}\int_{t^{q-2}}^{t^{q+1}}(t^{q-2}-s)^2\frac{\partial^{3} \bm u}{\partial t^{3}}ds.
\end{split}
\end{equation*}
For the other terms in the righthand side of \eqref{eq:R2split}, we have
\begin{equation}\label{eq:non2}
\begin{split}
\Big |\Big(\textbf{A}\big(B_k({\bm u}^q) & \cdot \nabla[\bm u(t^{q+1})-B_k(\bm u(t^q))]\big), -\Delta \bar{\bm e}_N^{q+1}+\tau_k \Delta \bar{\bm e}_N^{q}\Big) \Big | \\
& \le C\|B_k({\bm u}^q) \|_1 \|\bm u(t^{q+1})-B_k(\bm u(t^q))\|_2 \|-\Delta \bar{\bm e}_N^{q+1}+\tau_k \Delta \bar{\bm e}_N^{q}\|\\
& \le C(\epsilon)\|B_k({\bm u}^q)\|_1^2 \|\bm u(t^{q+1})-B_k(\bm u(t^q))\|_2^2+ \epsilon\|-\Delta \bar{\bm e}_N^{q+1}+\tau_k \Delta \bar{\bm e}_N^{q}\|^2\\
& \le  C(\epsilon)\delta t^{2k-1} \int_{t^{q+1-k}}^{t^{q+1}}\|\frac{\partial^{k} \bm u}{\partial t^{k}}(s)\|_2^2ds+2\epsilon\|\Delta \bar{\bm e}_N^{q+1}\|^2+2\epsilon\|\Delta \bar{\bm e}_N^{q}\|^2;
\end{split}
\end{equation}
Since $d=2$, we can use \eqref{eq:ineq2d} to obtain
\begin{equation}\label{eq:non3}
\begin{split}
\Big |\Big(& \textbf{A}\big( B_k({\bm e}^q) \cdot \nabla B_k({\bm e}^q) \big),  -\Delta \bar{\bm e}_N^{q+1}+\tau_k \Delta \bar{\bm e}_N^{q} \Big) \Big | \\
& \le C \|B_k(\bar{\bm e}^q)\|_1^{1/2}\|B_k(\bar{\bm e}^q)\|^{1/2}\|B_k(\bar{\bm e}^q)\|_2^{1/2}\|B_k(\bar{\bm e}^q)\|_1^{1/2}\| -\Delta \bar{\bm e}_N^{q+1}+\tau_k \Delta \bar{\bm e}_N^{q}\|\\
& \le C\|B_k({\bm e}^q)\|_1\|B_k({\bm e}^q)\|_2\| -\Delta \bar{\bm e}_N^{q+1}+\tau_k \Delta \bar{\bm e}_N^{q}\|\quad \text{(true in 2d and 3d)}\\
& \le C(\epsilon)\|B_k({\bm e}^q)\|_1^2\|B_k({\bm e}^q)\|_2^2+\epsilon\| -\Delta \bar{\bm e}_N^{q+1}+\tau_k \Delta \bar{\bm e}_N^{q}\|^2\\
& \le C(\epsilon)\|B_k({\bm e}^q)\|_1^2\|B_k({\bm e}^q)\|_2^2+2\epsilon\|\Delta \bar{\bm e}_N^{q+1}\|^2+2\epsilon\|\Delta \bar{\bm e}_N^{q}\|^2;
\end{split}
\end{equation}
Thanks to \eqref{eq:ineq2}, we have
\begin{equation}\label{eq:non4}
\begin{split}
\Big |\Big( &\textbf{A}\big(B_k(\bm u(t^q)) \cdot \nabla B_k(\bar{\bm e}^q) \big),  -\Delta \bar{\bm e}_N^{q+1}+\tau_k \Delta \bar{\bm e}_N^{q} \Big) \Big | \\
& \le C\|B_k(\bm u(t^q))\|_2 \|B_k({\bm e}^q)\|_1 \|-\Delta \bar{\bm e}_N^{q+1}+\tau_k \Delta \bar{\bm e}_N^{q}\|\\
& \le C(\epsilon)\|B_k(\bm u(t^q))\|_2^2 \|B_k({\bm e}^q)\|_1^2+ \epsilon\|-\Delta \bar{\bm e}_N^{q+1}+\tau_k \Delta \bar{\bm e}_N^{q}\|^2\\
& \le C(\epsilon)\|B_k({\bm e}^q)\|_1^2+2\epsilon\|\Delta \bar{\bm e}_N^{q+1}\|^2+2\epsilon\|\Delta \bar{\bm e}_N^{q}\|^2.
\end{split}
\end{equation}
On the other hand, we derive from
\eqref{eq:inequk} and \eqref{eq: prestepsk2} that
\begin{equation*}
\quad |\eta_k^q-1| \le 2^{k} C_0^{k}\, \delta t^{k}+2^{k}C_{\Pi}^{k}N^{k(2-m)},\quad \forall q\le n.
\end{equation*}
Note that $\bm u_N^q=\eta_k^q\bar{\bm u}_N^q$, we can estimate $\|B_k(\bm e^q)\|_1^2$ by
\begin{equation}\label{eq:Bkerror}
\begin{split}
\|B_k(\bm e^q)\|_1^2&=\|B_k(\bm u_N^q-\bar{\bm u}_N^q)+B_k(\bar{\bm e}_N^q)+B_k(\bm e_{\Pi}^q)\|_1^2\\
& \le CC_0^{2k} \delta t^{2k}+CC_{\Pi}^{2k}N^{2k(2-m)}+C\|B_k(\bar{\bm e}_N^q)\|_1^2+C\|\bm u(t^q)\|_m^2 N^{2-2m}.
\end{split}
\end{equation}
Combining \eqref{eq: errorhigh}-\eqref{eq:Bkerror} and dropping some unnecessary terms, we arrive at
\begin{equation}\label{eq:tobesum}
\begin{split}
\sum_{i,j=1}^{k}g_{ij}&(\nabla \bar{\bm e}_N^{q+1+i-k}, \nabla \bar{\bm e}_N^{q+1+j-k}) -\sum_{i,j=1}^{k}g_{ij}( \nabla \bar{\bm e}_N^{q+i-k}, \nabla \bar{\bm e}_N^{q+j-k})+\delta t(\frac{\nu}{2}-10\epsilon)\|\Delta \bar{\bm e}_N^{q+1}\|^2\\
& \le \delta t (\frac{\nu \tau_k^2 }{2}+10\epsilon) \|\Delta \bar{\bm e}_N^{q}\|^2+ C(\epsilon)\delta t\|B_k(\bar{\bm e}_N^q)\|_1^2+C(\epsilon)\delta t\|B_k(\bar{\bm e}_N^q)\|_1^2\|B_k({\bm e}_N^q)\|_2^2\\
& +C(\epsilon) \delta t^{2k} \int_{t^{q+1-k}}^{t^{q+1}}(\|\frac{\partial^{k} \bm u}{\partial t^{k}}(s)\|_2^2+\|\frac{\partial^{k+1}\bm u}{\partial t^{k+1}}(s)\|^2)ds\\
&+ C(\epsilon) C_0^{2k} \delta t^{2k+1}(1+\|B_k({\bm e}^q)\|_2^2)+\delta t C(\epsilon) C_{\Pi}^{2k}N^{2k(2-m)} (1+\|B_k({\bm e}^q)\|_2^2)\\
&+ \delta t C(\epsilon)\|\bm u(t^q)\|_m^2N^{2-2m}(1+\|B_k(\bm e^q)\|_2^2).
\end{split}
\end{equation}
Since $\tau_k<1$, we can choose $\epsilon$ small enough such that
\begin{equation}\label{eq:epsilon}
\frac{\nu}{2}-10\epsilon>\frac{\nu \tau_k^2 }{2}+10\epsilon+\frac{\nu(1-\tau_k^2)}{4},
\end{equation}
and then taking the sum of \eqref{eq:tobesum} on $q$ from $k-1$ to $n$, noting that $G=(g_{ij})$ is a  symmetric  positive definite matrix with minimum eigenvalue $\lambda_G$, we obtain:
\begin{equation}\label{eq:error3}
\begin{split}
 \lambda_G \|\nabla \bar{\bm e}_N^{n+1}\|^2&  + \frac{\delta t\nu(1-\tau_k^2)}{4} \sum_{q=0}^{n+1}\|\Delta \bar{\bm e}_N^{q}\|^2 \\
 & \le  \sum_{i,j=1}^{k}g_{ij} (\nabla \bar{\bm e}_N^{n+1+i-k}, \nabla \bar{\bm e}_N^{n+1+j-k}) + \frac{\delta t\nu(1-\tau_k^2)}{4} \sum_{q=0}^{n+1}\|\Delta \bar{\bm e}_N^{q}\|^2\\
& \le C \delta t \sum_{q=0}^{n} \|\bar{\bm e}_N^q\|_1^2(\|B_k({\bm e}^q)\|_2^2+1)\\
& +C \delta t^{2k} \Big(\int_{0}^{T} (\|\frac{\partial^{k}\bm u}{\partial t^{k}}(s)\|_2^2+\|\frac{\partial^{k+1}\bm u}{\partial t^{k+1}}(s)\|^2)ds +C_0^{2k}(T+ \delta t\sum_{q=0}^n \|B_k({\bm e}^q)\|_2^2)\Big)\\
&+ C \big(C_{\Pi}^{2k}N^{2k(2-m)}+N^{2-2m}\big) \big (T+\delta t \sum_{q=0}^n \|B_k({\bm e}^q)\|_2^2\big).
\end{split}
\end{equation}
Noting that \eqref{eq:H2sumbound} and \eqref{eq:H1boundex} imply $\delta t \sum_{q=0}^{n}\|B_k({\bm e}^{q})\|_2^2 <C_{H^2}$ for some constant $C_{H^2}$ depends only on the exact solution $\bm u$.
Applying the discrete Gronwall Lemma \ref{Gron2} to \eqref{eq:error3}, we obtain
\begin{equation}\label{eq:error4}
\begin{split}
&\|\bar{\bm e}_N^{n+1}\|_1^2+\delta t \sum_{q=0}^{n+1}\|\bar{\bm e}_N^{q}\|_2^2 \\
& \le C\exp \big(C_{H^2}+1) \big)  \delta t^{2k} \int_{0}^{T} (\|\frac{\partial^{k}\bm u}{\partial t^{k}}(s)\|_2^2+\|\frac{\partial^{k+1}\bm u}{\partial t^{k+1}}(s)\|^2)ds\\
& +C\exp \big(C_{H^2}+1) \big)
(\delta t^{2k}C_0^{2k}+C_{\Pi}^{2k}N^{2k(2-m)}+N^{2-2m})(T+C_{H^2})\\
& \le C_1 (1+C_0^{2k})\delta t^{2k}+C_1(C_{\Pi}^{2k}N^{2k(2-m)}+N^{2-2m}),
\end{split}
\end{equation}
where $C_1$ is independent of $\delta t$, $C_0$,  $C_{\Pi}$, and can be defined as
\begin{equation}\label{eq:C1}
C_1:= C\exp (C_{H^2}+1)\max \Big(\int_{0}^{T} (\|\frac{\partial^{k}\bm u}{\partial t^{k}}(s)\|_2^2+\|\frac{\partial^{k+1}\bm u}{\partial t^{k+1}}(s)\|^2)ds,\,1,\,T+C_{H^2} \Big).
\end{equation}
Therefore,  \eqref{eq:error4} implies
\begin{equation}\label{eq:errorbound}
\|\bar{\bm e}_N^{n+1}\|_1^2,\,\delta t \sum_{q=0}^{n+1}\|\bar{\bm e}_N^{q}\|_2^2 \le C_1 (1+C_0^{2k})\delta t^{2k}+C_1(C_{\Pi}^{2k}N^{2k(2-m)}+N^{2-2m}).
\end{equation}
Since $\bar{\bm e}^{q}=\bar{\bm e}^q_N+\bar{\bm e}^q_{\Pi}$, it follows from the triangle inequality that
\begin{equation}\label{eq:errorbound2}
\|\bar{\bm e}^{n+1}\|_1^2 \le C_1 (1+C_0^{2k})\delta t^{2k}+C_1(C_{\Pi}^{2k}N^{2k(2-m)}+N^{2-2m})+C N^{2(1-m)},
\end{equation}
and
\begin{equation}\label{eq:errorbound3}
\delta t \sum_{q=0}^{n+1}\|\bar{\bm e}^{q}\|_2^2 \le C_1 (1+C_0^{2k})\delta t^{2k}+C_1(C_{\Pi}^{2k}N^{2k(2-m)}+N^{2-2m})+C N^{2(2-m)}.
\end{equation}
Combining \eqref{eq:H1boundex}, \eqref{eq:errorbound2} and \eqref{eq:errorbound3}, we find that, under the condition \eqref{eq: dtcond1} and $m \ge 3$, we have
\begin{equation}\label{eq:ubarbound}
\begin{split}
\|\bar{\bm u}_N^{n+1}\|_1^2,\,\delta t \sum_{q=0}^{n+1}\|\bar{\bm u}_N^{q}\|_2^2 & \le C_1 (1+C_0^{2k}\frac{1}{2^{2k(k+2)}C_0^{2k^2}})+C_1(C_{\Pi}^{2k}2^{-4k(k +1)}C_{\Pi}^{-4k^2}+1)+C\\
&\le 4C_1+C:=\bar{C}.
\end{split}
\end{equation}

\textbf{Step 3: Estimate for $|1-\xi^{n+1}|$.} It follows from \eqref{eq: NSsavE2} that the equation for  $\{s^j \}$ can be written as
\begin{equation}\label{eq:serror}
s^{q+1}-s^{q}=\delta t \nu \big(\|\Delta \bm u(t^{q+1})\|^2-\frac{r^{q+1}}{E(\bar{\bm u}_N^{q+1})+1}\|\Delta \bar{\bm u}_N^{q+1}\|^2\big)+T_q,\,\,\forall q\le n,
\end{equation}
where $T_q$ is the truncation error
\begin{equation}\label{Tn}
T_q=r(t^q)-r(t^{q+1})+\delta t r_t(t^{q+1})=\int_{t^q}^{t^{q+1}}(s-t^q)r_{tt}(s)ds.
\end{equation}
Taking the sum of \eqref{eq:serror} for $q$ from 0 to $n$, and noting that $s^0=0$, we have
\begin{equation}\label{eq:serrorsum}
s^{n+1}=\delta t \nu \sum_{q=0}^n\big(\|\Delta \bm u(t^{q+1})\|^2-\frac{r^{q+1}}{E(\bar{\bm u}_N^{q+1})+1}\|\Delta \bar{\bm u}_N^{q+1}\|^2\big)+\sum_{q=0}^{n}T_q,
\end{equation}
We bound the righthand side of \eqref{eq:serrorsum} as follows. By direct calculation, we have
\begin{equation}\label{rtt}
r_{tt}=\int_{\Omega} ((\nabla\bm u)_t^2+ \nabla \bm u (\nabla \bm u)_{tt}) dx,
\end{equation}
then from \eqref{Tn}, we have
\begin{equation*}
|T_q| \le C \delta t \int_{t^q}^{t^{q+1}}|r_{tt}|ds \le C \delta t \int_{t^q}^{t^{q+1}} (\|\bm u_t\|_1^2+\|\bm u_{tt}\|_1^2) ds,\,\forall q \le n.
\end{equation*}
By triangular inequality,
\begin{equation}\label{eq: K1K2}
\begin{split}
\big |&\|\Delta \bm u(t^{q+1})\|^2-\frac{r^{q+1}}{E(\bar{\bm u}_N^{q+1})+1}\|\Delta \bar{\bm u}_N^{q+1}\|^2 \big |\\
& \le \|\Delta \bm u(t^{q+1})\|^2\big|1-\frac{r^{q+1}}{E(\bar{\bm u}_N^{q+1})+1} \big|+\frac{r^{q+1}}{E(\bar{\bm u}_N^{q+1})+1}\big|\|\Delta \bm u(t^{q+1})\|^2-\|\Delta \bar{\bm u}_N^{q+1}\|^2 \big|\\
&:=K^q_1+K^q_2.
\end{split}
\end{equation}
It follows from \eqref{eq:H1boundex} and Theorem \ref{stableThm} that
\begin{equation}\label{eq: K1}
\begin{split}
K^q_1 & \le C \big|1-\frac{r^{q+1}}{E(\bar{\bm u}_N^{q+1})+1} \big|\\
& =C \big|\frac{r(t^{q+1})}{E[\bm u(t^{q+1})]+1}-\frac{r^{q+1}}{E[\bm u(t^{q+1})]+1} \big|+C \big|\frac{r^{q+1}}{E[\bm u(t^{q+1})]+1}-\frac{r^{q+1}}{E(\bar{\bm u}_N^{q+1})+1} \big|\\
& \le C \big(|E[\bm u(t^{q+1})]-E(\bar{\bm u}_N^{q+1})|+|s^{q+1}|\big),\,\, \forall q\le n,
\end{split}
\end{equation}
and it follows from \eqref{eq:H1boundex} and Theorem \ref{stableThm} that
\begin{equation}\label{eq: K2}
\begin{split}
K^q_2 & \le C \big |\|\Delta \bar{\bm u}_N^{q+1}\|^2-\|\Delta \bm u(t^{q+1})\|^2 \big | \\
& \le C \|\Delta \bar{\bm u}_N^{q+1}-\Delta \bm u(t^{q+1})\|(\|\Delta \bar{\bm u}^{q+1}\|+\|\Delta \bm u(t^{q+1})\|)\\
& \le C \|\Delta \bar{\bm u}_N^{q+1}\|\|\Delta \bar{\bm e}^{q+1}\|+C\|\Delta \bar{\bm e}^{q+1}\|,\,\, \forall q\le n.
\end{split}
\end{equation}
We derive from the definition of $E(\bm u)$ that
\begin{equation}\label{errorE}
|E(\bm u(t^{q+1}))-E(\bar{\bm u}_N^{q+1})| \le \frac{1}{2}(\|\nabla \bm u(t^{q+1})\|+\|\nabla \bar{\bm u}_N^{q+1}\|)\|\nabla \bm u(t^{q+1})-\nabla \bar{\bm u}_N^{q+1}\| \le C \|\nabla \bar{\bm e}^{q+1}\|.
\end{equation}
It follows from \eqref{eq:errorbound3}, \eqref{eq:ubarbound} and the Cauchy-Schwarz inequality that
\begin{equation}\label{error6}
\begin{split}
\delta t &\sum_{q=0}^{n}\|\Delta \bar{\bm u}_N^{q+1}\|\|\Delta \bar{\bm e}^{q+1}\| \le \big(\delta t \sum_{q=0}^{n}\|\Delta \bar{\bm u}^{q+1}\|^2\delta t \sum_{q=0}^{n}\|\Delta \bar{\bm e}^{q+1}\|^2\big)^{1/2} \\
& \le C \sqrt{C_1 (1+C_0^{2k})\delta t^{2k}+C_1(C_{\Pi}^{2k}N^{2k(2-m)}+N^{2-2m})+ N^{2(2-m)}}.
\end{split}
\end{equation}
Now, we are ready to estimate $s^{n+1}$. Combining the estimates obtained above, \eqref{eq:serrorsum} leads to
\begin{equation}\label{eq:errorS}
\begin{split}
|s^{n+1}| &\le \delta t \nu \sum_{q=0}^n\big |\|\nabla \bm u(t^{q+1})\|^2-\frac{r^{q+1}}{E(\bar{\bm u}^{q+1})+1}\|\nabla \bar{\bm u}_N^{q+1}\|^2\big|+\sum_{q=0}^{n}|T^q|\\
& \le C \delta t \sum_{q=0}^n |s^{q+1}|+C \delta t \sum_{q=0}^n \|\bar{\bm e}^{q+1}\|_2+C\delta t \sum_{q=0}^{n}\|\Delta \bar{\bm u}_N^{q+1}\|\|\Delta \bar{\bm e}^{q+1}\|\\
& + C \delta t \int_{0}^{t^{n+1}}(\|\bm u_t\|_1^2+\|\bm u_{tt}\|_1^2) ds \\
& \le C\sqrt{C_1 (1+C_0^{2k})\delta t^{2k}+C_1(C_{\Pi}^{2k}N^{2k(2-m)}+N^{2-2m})+N^{2(2-m)}}\\
& +C \delta t \sum_{q=0}^n |s^{q+1}|+C\delta t.
\end{split}
\end{equation}
Finally, applying Lemma \ref{Gron1} on \eqref{eq:errorS} with $\delta t<\frac{1}{2C}$, we obtain the following estimate for $s^{n+1}$:
\begin{small}
\begin{equation}\label{eq:errorS2}
\begin{split}
|s^{n+1}| & \le C \exp((1-\delta t C)^{-1})\big(\sqrt{C_1 (1+C_0^{2k})\delta t^{2k}+C_1(C_{\Pi}^{2k}N^{2k(2-m)}+N^{2-2m})+ N^{2(2-m)}}+\delta t\big )\\
& \le C_2 \big(\sqrt{C_1 (1+C_0^{2k})\delta t^{2k}+C_1(C_{\Pi}^{2k}N^{2k(2-m)}+N^{2-2m})+N^{2(2-m)}}+\delta t \big)\\
& \le  C_2 \delta t^k \sqrt{C_1 (1+C_0^{2k})}+C_2  \sqrt{C_1(C_{\Pi}^{2k}N^{2k(2-m)}+N^{2-2m})+N^{2(2-m)}}+C_2\delta t,
\end{split}
\end{equation}
\end{small}
where $C_2:=C\exp(2)$ is independent of $\delta t$ and $C_0$.
then $\delta t<\frac{1}{2C}$ can be guaranteed by
\begin{equation}\label{cond3}
\delta t <\frac{1}{C_2}.
\end{equation}
Thanks to \eqref{eq:errorbound}, \eqref{eq: K1}, \eqref{errorE}, \eqref{eq:errorS2} and $m\ge 3$ , we have
\begin{equation}\label{eq:errorxi}
\begin{split}
|1-\xi^{n+1}| &\le  C \big(|E[\bm u(t^{n+1})]-E(\bar{\bm u}^{n+1})|+|s^{n+1}|\big)\\
& \le C( \|\nabla \bar{\bm e}^{n+1}\|+|s^{n+1}|)\\
& \le C \sqrt{C_1 (1+C_0^{2k})\delta t^{2k}+C_1(C_{\Pi}^{2k}N^{2k(2-m)}+N^{2-2m})+C N^{2(1-m)}}\\
& +  C_2 \delta t^k \sqrt{C_1 (1+C_0^{2k})}+C_2  \sqrt{C_1(C_{\Pi}^{2k}N^{2k(2-m)}+N^{2-2m})+N^{2(2-m)}}+C_2\delta t\\
& \le C_3 \delta t(\sqrt{1+C_0^{2k}}\delta t^{k-1}+1)+C_3N^{2-m}\big (\sqrt{C_{\Pi}^{2k}N^{(4-2m)(k-1)}+N^{-2}+1}\big ),
\end{split}
\end{equation}
where the constant $C_3$ is independent of $C_0$, $C_{\Pi}$ , $\delta t$ and $N$. Without loss of generality, we  assume $C_3>\max\{C_1, C_2,1\}$ to simplify the proof below.

For the cases $k=2,3,4,5$,  we choose $C_0=2C_3$ and $\delta t \le \frac1{1+C_0^{k}}$ to obtain
\begin{equation}\label{C0def}
C_3(\sqrt{1+C_0^{2k}}\delta t^{k-1}+1)\le C_3[(1+C_0^{k})\delta t+1] \le 2C_3=C_0,
\end{equation}
and  since $m \ge 3$, we can choose $C_{\Pi}=3C_3$ and $N \ge C_{\Pi}^k+1$ to obtain
\begin{equation}\label{Cpidef}
C_3\big (\sqrt{C_{\Pi}^{2k}N^{(4-2m)(k-1)}+N^{-2}+1}\big )\le C_3[C_{\Pi}^k N^{2-m}+2] \le 3C_3=C_{\Pi}.
\end{equation}

For the case $k=1$, since $\eta_1^{n+1}=1-(1-\xi^{n+1})^2$,  we choose $C_0=2C_3$ and $\delta t \le \frac{1}{1+C_0^2}$ so that
\begin{equation*}\label{C0def2}
C_3(\sqrt{1+C_0^{4}}\delta t+1)\le C_3[(1+C_0^{2})\delta t+1] \le  2C_3=C_0,
\end{equation*}
and since $m \ge 3$, we choose $C_{\Pi}=3C_3$ and $N\ge C_{\Pi}^2+1$ to obtain
\begin{equation}\label{Cpidef2}
C_3\big (\sqrt{C_{\Pi}^{4}N^{(4-2m)}+N^{-2}+1}\big )\le C_3[C_{\Pi}^2 N^{2-m}+2] \le 3C_3=C_{\Pi}.
\end{equation}
To summarize, combining the above with \eqref{eq:errorxi},  we derive from \eqref{eq:errorxi} that
$$|1-\xi^{n+1}|\le C_0 \delta t+C_{\Pi}N^{2-m}$$
 under the conditions
\begin{equation}\label{cond4}
\delta t \le \frac1{1+2^{k+2} C_0^{k+1}},\quad,N \ge {2^{k+2}C_{\Pi}^{k+1}+1} \quad 1\le k \le 5.
\end{equation}
 Note that the above implies \eqref{eq: dtcond1}, and  with  $C_3>\max\{C_1, C_2,1\}$, it also implies \eqref{cond3}. The induction process for \eqref{eq: prestepsk} is complete.
\medskip

We derive from \eqref{eq: NSsavE4} and \eqref{eq:ubarbound} that
\begin{equation}\label{eq: errorxin+1h}
 \|\bm u_N^{n+1}-\bar{\bm u}_N^{n+1}\|_1^2 \le |\eta_k^{n+1}-1|^2
\|\bar {\bm u}_N^{n+1}\|_1^2 \le |\eta_k^{n+1}-1|^2 {C},
\end{equation}
and
\begin{equation}\label{eq:H2error}
\begin{split}
\delta t\sum_{q=0}^n \|{\bm u}_N^{q+1}-\bar{\bm u}_N^{q+1}\|_2^2 & \le \delta t\sum_{q=0}^n |\eta_k^{q+1}-1|^2\|\bar{\bm u}_N^{q+1}\|_2^2 \\
& \le \max_{q}|\eta_k^{q+1}-1|^2\delta t\sum_{q=0}^n \|\bar{\bm u}_N^{q+1}\|_2^2\\
& \le \max_{q}|\eta_k^{q+1}-1|^2 {C}.
\end{split}
\end{equation}
On the other hand, we derive from \eqref{eq: prestepsk} that
\begin{subequations}\label{eq: errorxin+2h}
\begin{align}
&|\eta_1^{q+1}-1| \le 2^2 C_0^{2} \delta t^{2}+2^2 C_{\Pi}^2 N^{2(2-m)},\qquad \forall q \le n \quad k=1,\\
&|\eta_k^{q+1}-1| \le 2^k C_0^{k} \delta t^{k}+2^k C_{\Pi}^k N^{k(2-m)},\qquad \forall q \le n \quad k=2,3,4,5.
\end{align}
\end{subequations}
Therefore, we derive from \eqref{eq:errorbound2}, \eqref{eq:errorbound3}, \eqref{eq: errorxin+1h}, \eqref{eq:H2error},  \eqref{eq: errorxin+2h} and the triangle inequality that
\begin{equation*}
\|\bm e^{n+1}\|_1^2 \le \| \bar{\bm e}^{n+1}\|_1^2+ \|\bm u_N^{n+1}-\bar{\bm u}_N^{n+1}\|_1^2,
\end{equation*}
and
\begin{equation*}
\|\bm e^{q+1}\|_2^2 \le \| \bar{\bm e}^{q+1}\|_2^2+ \|\bm u_N^{q+1}-\bar{\bm u}_N^{q+1}\|_2^2, \quad  \forall q \le n,
\end{equation*}
 under the condition \eqref{cond4} on $\delta t$ and $N$. The proof is now complete since we already proved
 \eqref{eq:errorbound2} and  \eqref{eq:errorbound3}.
 \end{proof}

 Using exactly the same procedure above without the spatial discretization, we can prove the following result for the semi-discrete schemes \eqref{eq: NSsavb}.
 \begin{cor}\label{ThmNSb}
 Let $d=2$, $T>0$, $\bm u_0\in  \bm V\cap\bm H_p^2$  and $\bm u$ be the solution of \eqref{eq:NS}.
We assume that $\bar {\bm u}^i$ and $\bm u^i$ $(i=1,\cdots,k-1)$ are computed with a proper initialization procedure such that for $(i=1,\cdots,k-1)$,
\begin{equation*}\label{inikb}
\begin{split}
\|\bar {\bm u}^i-\bm u(t_i)\|_1,\;\| \bm u^i-\bm u(t_i)\|_1=O(\delta t^k);\quad
\|\bar {\bm u}^i-\bm u(t_i)\|_2,\;\| \bm u^i-\bm u(t_i)\|_2=O(\delta t^k),
\end{split} \quad i=1,2,3,4,5.
\end{equation*}
Let $\bar {\bm u}^{n+1}$ and $\bm u^{n+1}$ be computed with the $k-$th order scheme  \eqref{eq: NSsavb} $(1 \le k \le5)$, and
\begin{equation*}
 \eta_1^{n+1}=1-(1-\xi^{n+1})^{2}, \quad \eta_k^{n+1}=1-(1-\xi^{n+1})^{k}\; (k=2,3,4,5).
\end{equation*}
Then for  $n+1 \le T/\delta t$ and $\delta t \le \frac1{1+2^{k+2} C_0^{k+1}}$, we have
\begin{equation*}
\|\bar{\bm u}^{n}-\bm u(\cdot,t^n)\|_1^2,\,\| {\bm u}^{n}-\bm u(\cdot,t^n)\|_1^2 \le C \delta t^{2k},
\end{equation*}
and
\begin{equation*}
\delta t\sum_{q=0}^n \|\bar {\bm u}^{q+1}-\bm u(\cdot,t^{q+1})\|_2^2,\,\delta t\sum_{q=0}^n \|{\bm u}^{q+1}-\bm u(\cdot,t^{q+1})\|_2^2 \le C \delta t^{2k}.
\end{equation*}
where the constants $C_0$ and $C$ are dependent on $T,\, \Omega,$ the $k\times k$ matrix $G=(g_{ij})$ in Lemma \ref{lemmaODEH} and the exact solution $\bm u$, but are independent of $\delta t$.


\end{cor}

\subsection{Error analysis for the velocity in 3D}
In the three-dimensional case, it is no longer possible to obtain the global estimates \eqref{eq:uL2bound}, \eqref{eq:H1sumbound} and \eqref{eq:H2sumbound} as in the two-dimensional case. Instead, we shall derive  local estimates in analogy to the local existence of strong solution for the 3-D Navier-Stokes equations.

\begin{theorem}\label{ThmNS3D}
 Let $d=3$, $T>0$, $\bm u_0\in \bm V\cap \bm H_p^m$ with $m\ge 3$. We assume that \eqref{eq:NS}  admits a unique strong solution  $\bm u$ in $C([0,T];\bm H_p^1)\cap L^2(0,T;\bm H_p^2)$.
 We assume \eqref{inik} as in {\bf Theorem 2}, and let $\bar {\bm u}_N^{n+1}$ and $\bm u_N^{n+1}$ be computed using the $k$th-order scheme  \eqref{eq: NSsav} $(1 \le k \le5)$, and
\begin{equation*}
 \eta_1^{n+1}=1-(1-\xi^{n+1})^{2}, \quad \eta_k^{n+1}=1-(1-\xi^{n+1})^{k}\; (k=2,3,4,5).
\end{equation*}
Then, there exits $T_*>0$  such that for $0<T<T_*$, $n+1 \le T/\delta t$ and
$\delta t \le \frac1{1+2^{k+2} C_0^{k+1}}$, $N \ge {2^{k+2}C_{\Pi}^{k+1}+1} $, we have
\begin{equation}\label{eq:error3d}
\|\bar{\bm u}_N^{n}-\bm u(\cdot,t^n)\|_1^2,\,\| {\bm u}_N^{n}-\bm u(\cdot,t^n)\|_1^2 \le C \delta t^{2k}+C N^{2(1-m)},
\end{equation}
and
\begin{equation}\label{eq:error3d2}
\delta t\sum_{q=0}^n \|\bar {\bm u}_N^{q+1}-\bm u(\cdot,t^{q+1})\|_2^2,\,\delta t\sum_{q=0}^n \|{\bm u}_N^{q+1}-\bm u(\cdot,t^{q+1})\|_2^2 \le C \delta t^{2k}+C N^{2(2-m)},
\end{equation}
where the constants $C_0$, $C_{\Pi}$, $C$ are dependent on $T,\, \Omega,$ the $k\times k$ matrix $G=(g_{ij})$ in Lemma \ref{lemmaODEH} and the exact solution $\bm u$, but are independent of $\delta t$ and $N$.
\end{theorem}
\begin{proof}
 
The proof follows essentially the same procedure as the proof for {\bf Theorem \ref{ThmNS}}. However,  since we only has
the weak version of the stability in {\bf Theorem 1}   and \eqref{eq:ineq2d} is not valid   when $d=3$,   we can only get a local version of  \eqref{eq:uL2bound} and \eqref{eq:H1sumbound}. To simplify the presentation, we shall only point out below  the main differences with  the proof for {\bf Theorem \ref{ThmNS}}.

With  $\bm u_0\in \bm H_p^m$ and the existence of a unique strong solution  $\bm u$ in $C([0,T];\bm H_p^1)\cap L^2(0,T;\bm H_p^2)$, regularity results in \cite{Tema83,foias1989gevrey}  imply that   \eqref{regNS} is also valid in the three-dimensional case.

In \textbf{Step 1}, we still assume \eqref{eq: prestepsk2} holds and choose $\delta t$ and $N$ satisfies \eqref{eq: dtcond1}. Let $v_N=-\Delta \bar{\bm u}^{n+1}+\tau_k\Delta \bar{\bm u}^n$ in \eqref{eq: NSsavE1}, it follows from Lemma \ref{lemmaODEH} that
\begin{equation}\label{eq:bound3d}
\begin{split}
\sum_{i,j=1}^{k}g_{ij}&( \nabla \bar{\bm u}_N^{q+1+i-k}, \nabla \bar{\bm u}_N^{q+1+j-k}) -\sum_{i,j=1}^{k}g_{ij}( \nabla \bar{\bm u}_N^{q+i-k}, \nabla \bar{\bm u}_N^{q+j-k})\\
&+\|\sum_{i=0}^{k} \delta_i \nabla \bar{\bm u}_N^{q+1+i-k}\|^2
+\delta t \nu \|\Delta \bar{\bm u}_N^{q+1}\|^2\\
& = \delta t \nu (\Delta \bar{\bm u}_N^{q+1}, \tau_k \Delta \bar{\bm u}_N^{q}) +\delta t (\textbf{A}\big((B_k({\bm u}_N^q) \cdot \nabla) B_k({\bm u}_N^q)\big), -\Delta \bar{\bm u}_N^{q+1}+\tau_k \Delta \bar{\bm u}_N^{q}).
\end{split}
\end{equation}
We now bound the right hand side of \eqref{eq:bound3d}. Note that \eqref{eq: dtcond1} implies
\begin{equation*}
 \frac{1}{2}<1-(\frac{\delta t^{k-1}}{4}+\frac{N^{k(2-m)+1}}{4})\le |\eta_k^q| \le 1+\frac{\delta t^{k-1}}{4}+\frac{N^{k(2-m)+1}}{4}<2,\,\, \forall q\le n.
\end{equation*}
First, we have
\begin{equation}\label{eq:3dbound2}
|\delta t \nu (\Delta \bar{\bm u}_N^{q+1}, \tau_k \Delta \bar{\bm u}_N^{q})| \le \delta t \frac{\nu}{2}\|\Delta \bar{\bm u}_N^{q+1}\|^2+\delta t \frac{\nu\tau_k}{2}\|\Delta \bar{\bm u}_N^{q}\|^2.
\end{equation}
Next, it follows from \eqref{eq:ineq3d} that
\begin{equation}\label{eq:3dbound3}
\begin{split}
|(\textbf{A}\big(& (B_k({\bm u}_N^q) \cdot \nabla) B_k({\bm u}_N^q)\big), -\Delta \bar{\bm u}_N^{q+1}+\tau_k \Delta \bar{\bm u}_N^{q})| \\
&\le C\|B_k({\bm u}_N^q)\|_1\|B_k(\nabla {\bm u}_N^q)\|_{1/2}\|-\Delta \bar{\bm u}_N^{q+1}+\tau_k \Delta \bar{\bm u}_N^{q}\|\\
& \le C\|B_k({\bm u}_N^q)\|_1\|B_k( {\bm u}_N^q)\|_{1}^{1/2}\|B_k( {\bm u}_N^q)\|_{2}^{1/2}\|-\Delta \bar{\bm u}_N^{q+1}+\tau_k \Delta \bar{\bm u}_N^{q}\| \\
& \le C(\epsilon)\|B_k({\bm u}_N^q)\|_1^3\|B_k( {\bm u}_N^q)\|_{2}+\epsilon\|-\Delta \bar{\bm u}_N^{q+1}+\tau_k \Delta \bar{\bm u}_N^{q}\|^2\\
& \le C(\epsilon)\|B_k({\bm u}_N^q)\|_1^6+\epsilon \|B_k({\bm u}_N^q)\|_{2}^2+2\epsilon\|\Delta \bar{\bm u}_N^{q+1}\|^2+2\epsilon\|\Delta \bar{\bm u}_N^{q}\|^2.
\end{split}
\end{equation}
Now, combining \eqref{eq:bound3d}-\eqref{eq:3dbound3} and noting that $\bm u_N^q= \eta_k^q \bar{\bm u}_N^q$, we find after dropping some unnecessary terms that
\begin{equation}\label{eq:tobesum3d}
\begin{split}
\sum_{i,j=1}^{k}g_{ij}&(\nabla \bar{\bm u}_N^{q+1+i-k}, \nabla \bar{\bm u}_N^{q+1+j-k}) -\sum_{i,j=1}^{k}g_{ij}( \nabla \bar{\bm u}_N^{q+i-k}, \nabla \bar{\bm u}_N^{q+j-k})+\delta t(\frac{\nu}{2}-2\epsilon)\|\Delta \bar{\bm u}_N^{q+1}\|^2\\
& \le \delta t (\frac{\nu \tau_k }{2}+2\epsilon) \|\Delta \bar{\bm u}_N^{q}\|^2+ \epsilon \delta t\|B_k({\bm u}_N^q)\|_{2}^2+C(\epsilon)\delta t\|B_k({\bm u}_N^q)\|_1^6\\
& \le \delta t (\frac{\nu \tau_k }{2}+2\epsilon) \|\Delta \bar{\bm u}_N^{q}\|^2+ 2^2\epsilon \delta t  \|B_k(\bar {\bm u}_N^q)\|_{2}^2+ 2^6 C(\epsilon)\delta t\|B_k(\bar {\bm u}_N^q)\|_1^6
\end{split}
\end{equation}
Taking the sum of \eqref{eq:tobesum3d} for $q$ from $k-1$ to $n-1$, noting that $G=(g_{ij})$ is a symmetric positive definite  matrix with the minimum eigenvalue $\lambda_G$ and $\tau_k<1$, we can choose $\epsilon$ small enough such that:
\begin{equation*}
\begin{split}
 \lambda_G \|\bar{\bm u}_N^{n}\|_1^2&  + \frac{\delta t\nu(1-\tau_k)}{4} \sum_{q=0}^{n}\|\Delta \bar{\bm u}_N^{q}\|^2 \\
 & \le  \sum_{i,j=1}^{k}g_{ij} (\nabla \bar{\bm u}^{n+i-k}, \nabla \bar{\bm u}^{n+j-k}) + \frac{\delta t\nu(1-\tau_k)}{4} \sum_{q=0}^{n}\|\Delta \bar{\bm u}_N^{q}\|^2\\
& \le C \delta t \sum_{q=0}^{n-1} \|\bar{\bm u}_N^q\|_1^6 +M_0,
\end{split}
\end{equation*}
where $M_0>0$ is a constant only depends on $\bar{\bm u}_N^0,...,\bar{\bm u}_N^k,\, g_{ij}$. If we define $\phi$ as
$ \phi(x)=x^6 $
and let
\begin{equation}\label{eq:Tstar}
0<T_*<\int_{M_0}^{\infty}dz/\phi(z),
\end{equation}
then Lemma \ref{locallemma}
implies that there exist $C_*>0$ independent of $\delta t$ such that
\begin{equation}\label{eq:H2bound3d}
\|\bar{\bm u}_N^{n}\|_1^2 +  \delta t \sum_{q=0}^{n}\|\Delta \bar{\bm u}_N^{q}\|^2 \le C_*, \quad \forall n<T_*/\delta t.
\end{equation}
With \eqref{eq:H2bound3d} holds true, we can then prove \eqref{eq:error3d} and \eqref{eq:error3d2} by following  the same procedures in \textbf{Step 2} and \textbf{Step 3} in the proof of {\bf Theorem \ref{ThmNS}}.
\end{proof}

Similarly, we can prove the following result for the semi-discrete scheme \eqref{eq: NSsavb}.
\begin{cor}\label{ThmNS3Db} 
 Let $d=3$, $T>0$, $\bm u_0\in \bm V\cap \bm H_p^m$ with $m\ge 3$. We assume that \eqref{eq:NS}  admits a unique strong solution  $\bm u$ in $C([0,T];\bm H_p^1)\cap L^2(0,T;\bm H_p^2)$.
 We assume \eqref{inik} as in {\bf Theorem 2}, and let $\bar {\bm u}^{n+1}$ and $\bm u^{n+1}$ be computed using the $k$th-order schemes  \eqref{eq: NSsavb}, and
\begin{equation*}
 \eta_1^{n+1}=1-(1-\xi^{n+1})^{2}, \quad \eta_k^{n+1}=1-(1-\xi^{n+1})^{k}\; (k=2,3,4,5).
\end{equation*}
Then, there exits $T_*>0$  such that for $0<T<T_*$, $n+1 \le T/\delta t$ and
$\delta t \le \frac1{1+2^{k+2} C_0^{k+1}}$, $N \ge {2^{k+2}C_{\Pi}^{k+1}+1} $, we have
\begin{equation*}\label{eq:error3db}
\|\bar{\bm u}^{n}-\bm u(\cdot,t^n)\|_1^2,\,\| {\bm u}^{n}-\bm u(\cdot,t^n)\|_1^2 \le C \delta t^{2k},
\end{equation*}
and
\begin{equation*}\label{eq:error3d2b}
\delta t\sum_{q=0}^n \|\bar {\bm u}^{q+1}-\bm u(\cdot,t^{q+1})\|_2^2,\,\delta t\sum_{q=0}^n \|{\bm u}^{q+1}-\bm u(\cdot,t^{q+1})\|_2^2 \le C \delta t^{2k},
\end{equation*}
where $T_*$ is defined in \eqref{eq:Tstar}, the constants $C_0$, $C_{\Pi}$, $C$ are dependent on $T_*,\, \Omega,$ the $k\times k$ matrix $G=(g_{ij})$ in Lemma \ref{lemmaODEH} and the exact solution $\bm u$, but are independent of $\delta t$.
\end{cor}

\subsection{Error analysis for the pressure}
With the established error estimates for the velocity $\bm u$, the error estimate for  the pressure $p$ can be  derived directly from \eqref{prePoissont} or \eqref{prePoissonN}.

We denote
\[e_{pN}^n:= p_N^n- \Pi_N p(\cdot,t^n),\; e_{p\Pi}^n:=\Pi_N p(\cdot,t^n)-p(\cdot,t^n),\;\text{ and }   e_p^n =e_{pN}^n+e_{p\Pi}^n\].
\begin{theorem}\label{thmp}  Under the same assumptions as in  {\bf Theorem \ref{ThmNS} and Theorem \ref{ThmNS3D}},  we have
\begin{equation}\label{errorpL2}
\|p_N^{n+1}-p(\cdot, t^{n+1})\|^2 \le \left\{
\begin{array}{lr}
C \delta t^{2k}+C N^{2(1-m)}, \,\, \forall n\le T/\delta t, \qquad d=2,\\
C \delta t^{2k}+C N^{2(1-m)}, \,\, \forall n\le T_*/\delta t, \qquad d=3.
\end{array}
\right.
\end{equation}
and
\begin{equation}\label{errorp}
\delta t \sum_{q=0}^{n}\|\nabla (p_N^{n+1}-p(\cdot, t^{n+1}))\|^2 \le \left\{
\begin{array}{lr}
C \delta t^{2k}+C N^{2(2-m)}, \,\, \forall n\le T/\delta t, \qquad d=2,\\
C \delta t^{2k}+C N^{2(2-m)}, \,\, \forall n\le T_*/\delta t, \qquad d=3.
\end{array}
\right.
\end{equation}
where  $p_N^{n+1}$ is  computed from \eqref{prePoissonN},  $T_*$ is defined in \eqref{eq:Tstar} and $C$ is a constant independent of $\delta t$ and $N$.
\end{theorem}
\begin{proof}
From \eqref{prePoissonN}, we can write down the error equation for $p_N^{n+1}$ as
\begin{equation}\label{eq:errorp}
 \big( \nabla e_p^{q+1}, \nabla v_N \big)=\big(\bm u_N^{q+1} \cdot \nabla \bm u_N^{q+1}-\bm u(t^{q+1}) \cdot \nabla \bm u(t^{q+1}), \nabla v_N \big),\, \forall v_N \in S_N,\, \forall q+1\le n.
\end{equation}


To prove \eqref{errorpL2}, we  set $v_N=\Delta^{-1} e_{pN}^{q+1}$ in \eqref{eq:errorp} to obtain
\begin{equation}\label{eq:errorp2L2}
\begin{split}
\| e_{pN}^{q+1}\|^2 & = \Big(\bm u_N^{q+1} \cdot \nabla[\bm u_N^{q+1}-\bm u(t^{q+1})], \Delta^{-\frac{1}{2}} e_{pN}^{q+1} \Big)\\
& -\Big([\bm u(t^{q+1})-\bm u_N^{q+1}] \cdot \nabla \bm u(t^{q+1}), \Delta^{-\frac{1}{2}}  e_{pN}^{q+1} \Big)
\end{split}
\end{equation}
We can bound the righthand side of \eqref{eq:errorp2L2}  by using \eqref{eq:ineq2}, the stability result Theorem \ref{stableThm} and error analysis for the velocity, namely, we can obtain
\begin{equation}\label{eq:pnon1L2}
\begin{split}
\Big|\Big(\bm u_N^{q+1} \cdot \nabla[\bm u_N^{q+1}-\bm u(t^{q+1})], \Delta^{-\frac{1}{2}} e_{pN}^{q+1} \Big) \Big| & \le C(\epsilon)\|\bm u_N^{q+1}\|_1^2\|\|\bm e^{q+1}\|_1^2+\epsilon\|\nabla { e}_{pN}^{q+1}\|^2\\
& \le  C(\epsilon)(\delta t^{2k}+N^{2(1-m)})+\epsilon\|{ e}_{pN}^{q+1}\|^2;
\end{split}
\end{equation}
and
\begin{equation}\label{eq:pnon2L2}
\begin{split}
\Big|-\Big([\bm u(t^{q+1})-\bm u_N^{q+1}] \cdot \nabla \bm u(t^{q+1}), \Delta^{-\frac{1}{2}}  e_{pN}^{q+1} \Big)\Big| & \le C(\epsilon)\|\bm u(t^{q+1})\|_1^2\|\|\bm e^{q+1}\|_1^2+\epsilon\|\nabla { e}_{pN}^{q+1}\|^2\\
& \le  C(\epsilon)(\delta t^{2k}+N^{2(1-m)})+\epsilon\|{ e}_{pN}^{q+1}\|^2;
\end{split}
\end{equation}
Combining \eqref{eq:errorp2L2}-\eqref{eq:pnon2L2} with $\epsilon=\frac{1}{4}$ we obtain
\begin{equation}\label{eq:errorpreL2}
\|e_{pN}^{q+1}\|^2 \le C \delta t^{2k}+C N^{2(1-m)},\quad \forall q\le n.
\end{equation}
To prove \eqref{errorp}, we set $v_N=e_{pN}^{q+1}$ in \eqref{eq:errorp} to obtain
\begin{equation}\label{eq:errorp2}
\begin{split}
\| \nabla e_{pN}^{q+1}\|^2 & = \Big(\bm u_N^{q+1} \cdot \nabla[\bm u_N^{q+1}-\bm u(t^{q+1})], \nabla e_{pN}^{q+1} \Big)\\
& -\Big([\bm u(t^{q+1})-\bm u_N^{q+1}] \cdot \nabla \bm u(t^{q+1}), \nabla  e_{pN}^{q+1} \Big)
\end{split}
\end{equation}
Again, we can bound the righthand side of \eqref{eq:errorp2} in a similar fashion as in \eqref{eq:pnon1L2}-\eqref{eq:pnon2L2}, namely, we can obtain
\begin{equation}\label{eq:pnon1}
\begin{split}
\Big|\Big(\bm u_N^{q+1} \cdot \nabla[\bm u_N^{q+1}-\bm u(t^{q+1})], \nabla e_{pN}^{q+1} \Big) \Big| & \le C(\epsilon)\|\bm u_N^{q+1}\|_1^2\|\|\bm e^{q+1}\|_2^2+\epsilon\|\nabla { e}_{pN}^{q+1}\|^2\\
& \le  C(\epsilon)\|\bm e^{q+1}\|_2^2+\epsilon\|\nabla {e}_{pN}^{q+1}\|^2;
\end{split}
\end{equation}
and
\begin{equation}\label{eq:pnon2}
\begin{split}
\Big|-\Big([\bm u(t^{q+1})-\bm u_N^{q+1}] \cdot \nabla \bm u(t^{q+1}), \nabla e_{pN}^{q+1} \Big)\Big| & \le C(\epsilon)\|\bm u(t^{q+1})\|_2^2\|\|\bm e^{q+1}\|_1^2+\epsilon\|\nabla { e}_{pN}^{q+1}\|^2\\
& \le  C(\epsilon)(\delta t^{2k}+N^{2(1-m)})+\epsilon\|\nabla {e}_{pN}^{q+1}\|^2;
\end{split}
\end{equation}

Combining \eqref{eq:errorp2}-\eqref{eq:pnon2} with $\epsilon=\frac{1}{4}$, we obtain
\begin{equation}\label{eq:tobesumq}
\|\nabla e_{pN}^{q+1}\|^2 \le C \|{\bm e}^{q+1}\|_2^2+C \delta t^{2k}+C N^{2(1-m)},\quad \forall q \le n.
\end{equation}
Taking the sum of \eqref{eq:tobesum} for $q$ from 0 to $n$ and multiplying $\delta t$ on  both sides, we arrive at
\begin{equation}\label{eq:errorp3}
\delta t \sum_{q=0}^{n} \|\nabla e_{pN}^{q+1}\|^2 \le C \delta t \sum_{q=0}^{n}\|{\bm e}^{q+1}\|_2^2+C \delta t^{2k}+C N^{2(1-m)}.
\end{equation}
Now, with the estimates on $\|{\bm e}^n\|_2^2$ in Theorem \ref{ThmNS} or Theorem \ref{ThmNS3D}, \eqref{eq:errorp3} leads to
\begin{equation}\label{eq:errorpreL2b}
\delta t \sum_{q=0}^{n} \|\nabla e_{pN}^{q+1}\|^2 \le C \delta t^{2k}+C N^{2(2-m)}.
\end{equation}
Finally, we can obtain \eqref{errorpL2} and \eqref{errorp} from \eqref{eq:errorpreL2}, \eqref{eq:errorpreL2b} and
\begin{equation*}
\|\nabla e_{p\Pi}^q\|^2 \le C N^{2(1-m)}.
\end{equation*}
\end{proof}

Similarly, we can derive the following results for the semi-discrete scheme \eqref{eq: NSsavb}.

\begin{cor}\label{thmpb}  Under the same assumptions as in  {\bf Corollary 1  and Corollary 2},  we have
\begin{equation*}
\|p^{n+1}-p(\cdot,t^{n+1})\|^2 \le \left\{
\begin{array}{lr}
C \delta t^{2k}, \,\, \forall n\le T/\delta t, \qquad d=2,\\
C \delta t^{2k}, \,\, \forall n\le T_*/\delta t, \qquad d=3.
\end{array}
\right.
\end{equation*}
and
\begin{equation*}
\delta t \sum_{q=0}^{n}\|\nabla (p^{q+1}-p(\cdot,t^{n+1}))\|^2 \le \left\{
\begin{array}{lr}
C \delta t^{2k}, \,\, \forall n\le T/\delta t, \qquad d=2,\\
C \delta t^{2k}, \,\, \forall n\le T_*/\delta t, \qquad d=3.
\end{array}
\right.
\end{equation*}
where  $p^{n+1}$ is  computed from \eqref{prePoissont},  $T_*$ is defined in \eqref{eq:Tstar} and $C$ is a constant independent of $\delta t$.
\end{cor}

\section{Concluding remarks}
We considered numerical approximation of the  incompressible Navier-Stokes equations with periodic boundary conditions for which the pressure can be explicitly eliminated, allowing us to construct very efficient IMEX type schemes using Fourier-Galerkin approximation in space. Our
 high-order  semi-discrete-in-time and fully discrete  IMEX  schemes   are  based on a scalar auxiliary variable (SAV) approach which enables us to derive uniform bounds for the numerical solution without any restriction on time step size.   We also take advantage of an additional energy dissipation law \eqref{eq:diss2d}, which is only valid for the two-dimensional Navier-Stokes equations with periodic boundary conditions, leading to a uniform bound in $H^1$-norm, instead of the usual $L^2$-norm.
 By using these uniform bounds and a delicate induction process,  we derived  global  error estimates in $l^\infty(0,T;H^1)\cap l^2(0,T;H^2)$ in the two dimensional case as well as local  error estimates in $l^\infty(0,T;H^1)\cap l^2(0,T;H^2)$ in the three dimensional case for our  semi-discrete-in-time and fully discrete  IMEX  schemes up to fifth-order. We also validated our schemes with manufactured exact solutions and with the double shear layer problem. Our  numerical results for the double shear layer problem indicate that  the SAV approach can effectively prevent numerical solution from blowing up, and that    higher-order schemes are preferable for flows with complex structures such as the double shear layer problem with thin layers.

To the best of our knowledge, our numerical schemes are the first unconditionally stable high-order
 IMEX type schemes for Navier-Stokes equations  without any restriction on time step size, and our error estimates
 are the first for any IMEX type scheme  for the Navier-Stokes equations  in the three-dimensional case.

While the stability results can be extended to similar schemes for the Navier-Stokes equations with non-periodic boundary conditions, it is non trivial to carry out  the corresponding error analysis which will be left as  a subject of future endeavor.

\bibliographystyle{plain}
\bibliography{bib_sav_error}
\end{document}